\documentclass[twocolumn]{autart}    

\renewcommand{\natural}{{\mathbb{N}}}

\newcommand{\real}{\ensuremath{\mathbb{R}}}

\newcommand{\until}[1]{\{1,\dots, #1\}}
\newcommand{\subscr}[2]{#1_{\textup{#2}}}
\newcommand{\supscr}[2]{#1^{\textup{#2}}}
\newcommand{\setdef}[2]{\{#1 \; | \; #2\}}

\newcommand{\diag}[1]{\operatorname{diag}(#1)}

\newcommand{\spn}{\operatorname{span}}

\newcommand{\vect}[1]{\boldsymbol{#1}}

\renewcommand{\epsilon}{\varepsilon}

\newcommand{\Pc}{{\mathcal{P}}}

\usepackage{psfrag}

\usepackage{graphicx,epsfig,float,wrapfig}
\usepackage{epstopdf}
\usepackage{float}
\usepackage{epsfig}
\usepackage{url}
\usepackage{amssymb,amsmath,amsfonts,layout}
\usepackage[mathscr]{euscript}
\usepackage{makeidx}
\usepackage{blkarray,multirow}
\usepackage{cite}
\usepackage{enumerate}
\usepackage{algpseudocode}
\usepackage{algorithm}
\usepackage{enumitem,indentfirst}

\let\theoremstyle\relax

\usepackage{amsthm}
\usepackage{bbm}
\usepackage{xspace}
\usepackage [english]{babel}
\usepackage [autostyle, english = american]{csquotes}
\usepackage{xcolor}

\newcommand{\nulo}{\operatorname{null}}

\newcommand{\dom}{\operatorname{dom}}

\algnewcommand{\Inputs}[1]{%
	\State \textbf{Input:}
	\Statex \hspace*{\algorithmicindent}\parbox[t]{.8\linewidth}{\raggedright #1}
}
\algnewcommand{\Initialize}[1]{%
	\State \textbf{Initialize:}
	\Statex \hspace*{\algorithmicindent}\parbox[t]{.8\linewidth}{\raggedright #1}
}
\algnewcommand{\Outputs}[1]{%
	\State \textbf{Output:}
	\Statex \hspace*{\algorithmicindent}\parbox[t]{.8\linewidth}{\raggedright #1}
}

\newcommand{\distnewton}{\textsc{distributed approx-Newton}\xspace}
\newcommand{\dana}{\textsc{DANA}\xspace}
\newcommand{\distnewtondisc}{\textsc{discrete distributed approx-Newton}\xspace}
\newcommand{\danad}{\textsc{DANA-D}\xspace}
\newcommand{\distnewtoncont}{\textsc{continuous distributed approx-Newton}\xspace}
\newcommand{\danac}{\textsc{DANA-C}\xspace}
\newcommand{\distgrad}{\textsc{distributed gradient descent}\xspace}
\newcommand{\dgd}{\textsc{DGD}\xspace}

\newcommand{\longthmtitle}[1]{\mbox{}{\bf \textit{(#1).}}}
\newtheorem{thrm}{Theorem}
\newtheorem{lem}{Lemma}
\newtheorem{cor}{Corollary}
\newtheorem{claim}{Claim}

\newtheorem{assump}{Assumption}
\newtheorem{defn}{Definition}
\newtheorem{prob}{Problem}

\newtheorem{rem}{Remark}

\DeclareMathOperator{\N}{\mathcal{N}}

\DeclareMathOperator{\E}{\mathcal{E}}

\DeclareMathOperator{\G}{\mathcal{G}}
\DeclareMathOperator{\Ls}{\mathscr{L}}

\newcommand{\nodes}{\N}
\newcommand{\ones}[1][]{\vect{1}_{#1}}
\newcommand{\zeros}[1][]{\vect{0}_{#1}}

\newcommand{\setdefB}[2]{\Big\{#1 \; | \; #2\Big\}}

\newcommand\oprocendsymbol{\hbox{$\bullet$}}
\newcommand\oprocend{\relax\ifmmode\else\unskip\hfill\fi\oprocendsymbol}

\begin{document}

\begin{frontmatter}

\title{Distributed Approximate Newton Algorithms and Weight Design for Constrained
	Optimization}

\author{Tor Anderson\quad}
\author{Chin-Yao Chang\quad}
\author{Sonia Mart{\'\i}nez}

\thanks{A preliminary version appeared at the 2017 Conference on Control Technology and Applications~\cite{TA-CYC-SM:17a}.	Tor Anderson and Sonia
	Mart{\'\i}nez are with the Department of Mechanical and Aerospace
	Engineering, University of California, San Diego, CA, USA, and Chin-Yao Chang is at the National Renewable Energy Laboratory in Golden, CO, USA. Email: {\small {\tt \{tka001, soniamd\}@eng.ucsd.edu}}, {\small {\tt chinyao.chang@nrel.gov}}.
	This research was supported by the Advanced Research Projects Agency - Energy under the NODES program, Cooperative Agreement DE-AR0000695.}


\begin{abstract}                          Motivated by economic dispatch and linearly-constrained resource
	allocation problems, this paper proposes a class of novel
	$\distnewton$ algorithms that approximate the standard Newton
	optimization method. We first develop the notion of an optimal edge
	weighting for the communication graph over which agents implement
	the second-order algorithm, and propose a convex approximation for
	the nonconvex weight design problem.
	We next build on the optimal weight design to develop a
	\distnewtondisc algorithm which converges linearly
	to the optimal solution for economic dispatch problems with unknown
	cost functions and relaxed local box constraints. For the full
	box-constrained problem, we develop a \distnewtoncont algorithm
	which is inspired by first-order saddle-point methods and rigorously
	prove its convergence to the primal and dual optimizers. A main
	property of each of these distributed algorithms is that they only
	require agents to exchange constant-size communication messages,
	which lends itself to scalable implementations. Simulations
	demonstrate that the $\distnewton$ algorithms with our weight design
	have superior convergence properties compared to existing weighting
	strategies for first-order saddle-point and gradient descent
	methods.
\end{abstract}

\end{frontmatter}

\section{Introduction}
\textit{Motivation.}  Networked systems endowed with distributed,
multi-agent intelligence are becoming pervasive in modern
infrastructure systems such as power, traffic, and large-scale
distribution networks. However, these advancements lead to new
challenges in the coordination of the multiple agents operating the
network, which are mindful of the network dynamics, and subject to
partial information and communication constraints. 
To this end, distributed convex optimization is a rapidly emerging
field which seeks to develop useful algorithms to manage network
resources in a scalable manner. 
Motivated by the rapid emergence of distributed energy resources, a
problem that has recently gained large attention is that of economic
dispatch.  In this problem, a total net load constraint must be
satisfied by a set of generators which each have an associated cost of
producing electricity. 
However, the existing distributed techniques to solve this problem are
often limited by rate of convergence.  Motivated by this,
here 
we investigate the design of topology
weighting strategies that build on the Newton method and lead to
improved convergence rates. \\
\textit{Literature Review.}  The Newton method for minimizing a
real-valued multivariate objective function is well characterized for
centralized contexts in~\cite{SB-LV:04}.  Another centralized method
for solving general constrained convex problems by seeking the
saddle-point of the associated Lagrangian is developed
in~\cite{AC-EM-SHL-JC:18-tac}. This method, which implements a
saddle-point dynamics
is attractive because its convergence properties can be established.
Other first-order or primal-dual based methods for approaching distributed optimization include~\cite{EM-CZ-SL:17,TD-CB:17,RC-GN:13,SHM-MJ:18}.
However, these methods typically do not incorporate second-order
information of the cost function, which compromises convergence speeds.  The notion of computing an
\emph{approximate} Newton direction in distributed contexts has gained
popularity recently, such as~\cite{AM-QL-AR:17}
and~\cite{EW-AO-AJ:13P1,EW-AO-AJ:13P2}.
In the former work, the authors propose a method which uses the Taylor
series expansion for inverting matrices. However, it assumes that each
agent keeps an estimate of the entire decision variable, which does
not scale well in problems where this variable dimension is equal to
the number of agents in the network. Additionally, the optimization is
unconstrained, which helps to keep the problem decoupled but is
narrower in scope.  The latter works pose a separable optimization
with an equality constraint characterized by the incidence matrix. The
proposed method may be not directly applied to networks with
constraints that involve the information of all agents.  The
papers~\cite{DJ-JX-JM-14,FZ-DV-AC-GP-LS:16,RC-GN-LS-DV:15}
incorporate multi-timescaled dynamics together with a dynamic
consensus step to speed up the convergence of the agreement
subroutine.  These works only consider uniform edge weights, while
sophisticated design of the weighting may improve the
convergence. In~\cite{LX-SB:06}, the Laplacian weight design problem
for separable resource allocation is approached from a $\distgrad$
perspective. Solution post-scaling is also presented, which can be
found similarly in~\cite{MM-RT:07} and~\cite{YS:03} for improving the
convergence of the Taylor series expression for matrix inverses.
In~\cite{SYS-MA-LEG:10}, the authors consider edge weight design to
minimize the spectrum of Laplacian matrices. However, in the Newton
descent framework, the weight design problem formulates as a nonconvex
bilinear problem, which is challenging to solve.  Overall, the current
weight-design techniques that are computable in polynomial time are
only mindful of first-order algorithm dynamics. A second-order
approach has its challenges, which manifest themselves in a bilinear
design problem and more demanding communication requirements, but
using second-order information is more heedful of the problem geometry
and leads to faster convergence speeds.

\textit{Statement of Contributions.}  In this paper, we propose a
novel framework to design a weighted Laplacian matrix that is used in
the solution to a multi-agent optimization problem via sparse
approximated Newton algorithms. Motivated by economic dispatch, we
start by formulating a separable resource allocation problem subject
to a global linear constraint and local box constraints, and then
derive an equivalent form without the global constraint by means of a
Laplacian matrix, which is well suited for a distributed framework. We
use this to motivate weighting design of the elements of the Laplacian
matrix and formulate this problem as a bilinear optimization. We
develop a convex approximation of this problem whose solution can be
computed offline in polynomial time. A bound on the \emph{best-case}
solution of the original bilinear problem is also given.

We aim to bridge the gap between classic Newton and $\distnewton$
methods. To do this, we first relax the box constraints and develop a
class of constant step-size discrete-time algorithms.  The Newton
step associated with the unconstrained optimization problem do not
inherit the same sparsity as the distributed communication network. To
address this issue, we consider approximations based on a Taylor
series expansion, where the first few terms inherit certain level of
sparsity as prescribed by  the Laplacian matrix.  We analyze
the approximate algorithms and show their convergence for any
truncation of the series expansion.

We next study the original problem with local box constraints, which
has never been considered in the framework of a distributed Newton
method,
and present a novel continuous-time $\distnewton$ algorithm. The
convergence of this algorithm to the optimizer is rigorously studied
and we give an interpretation of the convergence in the Lyapunov
function sense. Furthermore, through a formal statement of the
proposed $\distnewton$ algorithm (or $\dana$), we find several
interesting insights on second-order distributed methods. We compare
the results of our design and algorithm to a generic weighting design
of $\distgrad$ ($\dgd$) implementations in
simulation. Our weighting design shows superior convergence to $\dgd$.

\textit{Organization.} The rest of the paper is organized as
follows. Section~\ref{sec:prelim} introduces the notations and
fundamentals used in this paper. We formulate the optimal resource
allocation problem in
Section~\ref{sec:prob-statement}. Section~\ref{sec:laplacian-design}
proposes the optimal graph weighting design for a second order method
and develops a convex approximation to compute a satisfactory
solution. In Section~\ref{sec:discrete-red}, we propose a distributed
algorithm that approximates the Newton step in solving the
optimization. Section~\ref{sec:sims-discuss} demonstrates the
effectiveness of the proposed algorithm. We conclude the paper in
Section~\ref{sec:conclusion}.

\section{Preliminaries}
\label{sec:prelim}
This section compiles notation and presents a few results that will be
used in the sequel.
\subsection{Notation}
Let $\real$ and $\real_{+}$ denote the set of real and positive real
numbers, respectively, and let $\natural$ denote the set of natural
numbers.  For a vector ${x \in \real^n}$, we denote by $x_i$ the
$\supscr{i}{th}$ entry of $x$.  For a matrix ${A \in\real^{n \times
		m}}$, we write $A_i$ as the $\supscr{i}{th}$ row of $A$ and
$A_{ij}$ as the element in the $\supscr{i}{th}$ row and
$\supscr{j}{th}$ column of $A$. Discrete time-indexed variables are
written as $x^k$, where $k$ denotes the current time step. The
transpose of a vector or matrix is denoted by $x^\top$ and $A^\top$,
respectively. We use the shorthand notations ${\mathbf{1}_n = [1,
	\dots, 1]^\top \in \real^n}$,
${\zeros[n] = [0, \dots, 0]^\top \in \real^n}$, and $I_n$
to denote the ${n \times n}$ identity matrix. The standard inner
product of two vectors $x,y \in \real^n$ is written $\langle x,
y\rangle$, and $x \perp y$ indicates $\langle x, y\rangle = 0$. For a
real-valued function ${f : \real^n \rightarrow \real}$, the gradient
vector of $f$ with respect to $x$ is denoted by ${\nabla_x f(x)}$ and
the Hessian matrix with respect to $x$ by ${\nabla_{xx} f(x)}$.  The
positive (semi) definiteness and negative (semi) definiteness of a
matrix ${A \in \real^{n \times n}}$ is indicated by ${A \succ 0}$ and
${A \prec 0}$ (resp. ${A \succeq 0}$ and ${A \preceq 0}$). The same
symbols are used to indicate componentwise inequalities on vectors of
equal sizes. The set of eigenvalues of a symmetric matrix ${A \in
	\real^{n \times n}}$ is ordered as $\mu_1(A) \leq \dots \leq \mu_n(A)$
with
associated eigenvectors $v_1, \dots , v_n \in \real^n$.  An orthogonal
matrix ${T \in \real^{n \times n}}$ has the property ${T^\top T = T
	T^\top = I_n}$ and ${T^\top = T^{-1}}$. For a finite set
$\mathcal{S}$, $\vert \mathcal{S} \vert$ is the cardinality of the
set. The uniform distribution on the interval $[a, b]$ is indicated by
$\mathcal{U} [a, b]$. We define the projection
\begin{equation*}
\left[ u \right]^+_v := \begin{cases}
u, & v > 0, \\
\text{max}\{0,u\}, & v \leq 0.
\end{cases}
\end{equation*}
\subsection{Graph Theory}
A network of agents is represented by a graph $\G =
(\nodes,\mathcal{E})$, assumed undirected, with a node set $\nodes =
\until{n}$ and edge set $\mathcal{E} \subseteq \nodes \times
\nodes$. The edge set $\mathcal{E}$ has elements $(i,j) \in
\mathcal{E}$ for $j \in \nodes_i$, where $\nodes_i \subset \nodes$ is
the set of neighbors of agent $i \in \nodes$.  The union of neighbors
to each agent $j \in \nodes_i$ are the 2-hop neighbors of agent $i$,
and denoted by $\nodes_i^2$. More generally, $\nodes_i^p$, or set of
$p$-hop neighbors of $i$, is the union of neighbors of agents in
$\nodes_i^{p-1}$. We consider \emph{weighted} edges for the sake of defining the graph Laplacian; the role of edge weightings and the associated design problem in this paper is described in Section~\ref{sec:laplacian-design}. The graph $\G$ then has a \emph{weighted} Laplacian ${L
	\in \real^{n \times n}}$ defined as
\begin{equation*}
L_{ij} =
\begin{cases}
-w_{ij}, & j \in \nodes_i, j \neq i, \\
w_{ii}, & j = i, \\
0, & \text{otherwise},
\end{cases}
\end{equation*}
with weights $w_{ij} = w_{ji} > 0, \forall j \neq i$, and total
incident weight $w_{ii}$ on~$i\in \nodes$, ${w_{ii} = \sum_{j \in
		\nodes_i} w_{ij}}$. Evidently, $L$ has an eigenvector ${v_1 =
	\ones[n]}$ with an associated eigenvalue ${\mu_1 = 0}$, and ${L
	= L^\top \succeq 0}$. The graph is connected i.f.f. $0$ is a
simple eigenvalue, i.e.~$0 = \mu_1 < \mu_2 \leq \dots \leq \mu_n$.

The Laplacian $L$ can be written via its incidence matrix
${E \in \{-1,0,1\}^{\vert \mathcal{E} \vert \times n}}$ and a diagonal
matrix ${X \in \real_{+}^{\vert \mathcal{E} \vert \times \vert
		\mathcal{E} \vert}}$ whose entries are weights $w_{ij}$. Each
row of $E$ is associated with an edge $(i,j)$ whose $\supscr{i}{th}$
element is $1$, $\supscr{j}{th}$ element is $-1$, and all other
elements zero.  Then, $L = E^\top X E$.

\subsection{Schur Complement}
The following lemma will be used in the sequel.
\begin{lem}
	\cite{FZ:05}\longthmtitle{Matrix Definiteness via Schur
		Complement}\label{lem:schur-comp}
	Consider a symmetric matrix $M$ of the form
	\begin{equation*}
	M = \begin{bmatrix}
	A & B \\ B^\top & C
	\end{bmatrix}.
	\end{equation*}
	If $C$ is invertible, then the following properties hold: \\
	(1) $M \succ 0$ if and only if $C \succ 0$ and $A - BC^{-1}B^\top \succ 0$. \\
	(2) If $C \succ 0$, then $M \succeq 0$ if and only if $A - BC^{-1}B^\top \succeq 0$.
\end{lem}
\subsection{Taylor Series Expansion for Matrix Inverses} \label{ssec:Taylor}

A full-rank matrix $A \in \real^{n \times n}$ has a matrix inverse,
$A^{-1}$, which is characterized by the relation $AA^{-1} = I_n$. In
principle, it is not straightforward to compute this inverse via a
distributed algorithm.  However, if the eigenvalues of $A$ satisfy
$\vert 1 - \mu_i(A) \vert < 1, \forall \,i \in \mathcal{N}$, then we
can employ the Taylor expansion to compute its inverse~\cite{GS:98}:
\begin{align*}
A^{-1} = \sum_{p=0}^\infty (I_n - A)^p.
\end{align*}

To quickly see this holds, substitute $B = I_n - A$, multiply both sides by
$I_n - B$ and reason with $\lim_{p\rightarrow\infty}$.
Note
that, if the sparsity structure of $A$ represents a network topology,
then traditional matrix inversion techniques such as Gauss-Jordan
elimination still necessitate all-to-all communication. However,
agents can communicate and compute locally to obtain each term in the
previous expansion.  If $A$ is normal, it can be seen via the diagonalization of $I_n -
A$ that the terms of the sum become small as $p$ increases due to the
assumption on the eigenvalues of $A$~\cite{SF-AI-LS:03}. The convergence of these terms
is exponential and limited by the slowest converging mode,
i.e.~$\max{\vert 1 - \mu_i(A) \vert}$.

We can compute an approximation of $A^{-1}$ in finite steps by
computing and summing the terms up to the $\supscr{q}{th}$ power.  We
refer to this approximation as a \textit{q-approximation} of $A^{-1}$.

\section{Problem Statement} \label{sec:prob-statement} Motivated by
the economic dispatch problem, in this section we pose the separable
resource allocation problem that we aim to solve distributively. We
reformulate it as an unconstrained optimization problem whose decision
variable is in the span of the graph Laplacian, and motivate the
characterization of a second-order Newton-inspired method.

Consider a group of agents $\nodes$, indexed by $i \in \nodes$, and a
communication topology given by $\G$. Each agent is associated with a
local convex cost function $f_i : \real \rightarrow \real$. These
agents can be thought of as generators in an electricity market, where
each function argument $x_i \in \real$, $i \in \nodes$ represents the
power that agent $i$ produces at a cost characterized by $f_i$. The
economic dispatch problem aims to satisfy a global load-balancing
constraint $\sum_{i=1}^n x_i = d$ for minimal global cost $f : \real^n
\rightarrow \real$, where $d$ is the total demand. In addition, each
agent is subject to a local linear box constraint on its decision
variable given by the interval
$[\underline{x}_i,\overline{x}_i]$. Then, the economic dispatch
optimization problem is stated as:
\begin{subequations} \label{eq:p1-opt}
	\begin{align}
	\Pc 1: \ & \underset{x}{\text{min}}
	& & f(x) = \sum_{i=1}^n f_i(x_i) \label{eq:p1-cost}\\
	& \text{subject to}
	& & \sum_{i=1}^n x_i = d, \label{eq:p1-resource-const}\\
	&&& \underline{x}_i \leq x_i \leq \overline{x}_i, \quad i =
	\until{n}. \label{eq:p1-box-const}
	\end{align}
\end{subequations}
Distributed optimization algorithms based on a gradient descent
approach to solve $\Pc 1$ are available~\cite{MZ-SM:15-book}.
However, by only taking into account first-order information of the
cost functions, these methods tend to be inherently slow. As for a
Newton (second-order) method, the constraints make the computation of
the descent direction non-distributed. To see this, consider
only~\eqref{eq:p1-cost}--\eqref{eq:p1-resource-const}.  Recall the
unconstrainted Newton step defined as $\subscr{x}{nt} := -\nabla_{xx}
f(x)^{-1}\nabla_x f(x)$, see e.g.~\cite{SB-LV:04}. In this context,
the equality constraint can be eliminated by imposing ${x_n = d -
	\sum_{i=1}^{n-1} x_i}$. Then, \eqref{eq:p1-cost} becomes ${f(x) =
	\sum_{i=1}^{n-1} f_i(x_i) + f_n(d - \sum_{i=1}^{n-1} x_i)}$. In
general, the resulting Hessian $\nabla_{xx} f(x)$ is fully populated
and its inverse requires all-to-all communication among agents in
order to compute the second-order descent direction. If we
additionally consider~\eqref{eq:p1-box-const}, interior point methods
are often employed, such as introducing a log-barrier function to the
cost in~\eqref{eq:p1-cost}\cite{SB-LV:04}. The value of the
log-barrier parameter is updated online to converge to a
feasible solution, which exacerbates the non-distributed nature of
this approach. This motivates the design of distributed Newton-like
methods which are cognizant
of~\eqref{eq:p1-resource-const}--\eqref{eq:p1-box-const}.

We eliminate \eqref{eq:p1-resource-const} by introducing a 
network topology as encoded by a Laplacian matrix $L$ associated with $\G$ 
and an initial condition $x^0 \in \real^n$ with some
assumptions.
\begin{assump}\longthmtitle{Undirected and Connected Graph} 	\label{ass:conn-graph}
	The weighted graph characterized by $L$ is undirected and connected,
	i.e. $L = L^\top$ and $0$ is a simple eigenvalue of $L$.
\end{assump}
\begin{assump}\longthmtitle{Feasible Initial
		Condition} \label{ass:initial}
	The initial state $x^0$ satisfies~\eqref{eq:p1-resource-const}, i.e.
	\begin{equation*}
	\sum_{i=1}^n x_i^0 = d.   
	\end{equation*}
\end{assump}
If the problem context does not lend itself well to satisfying
Assumption~\ref{ass:initial}, there is a distributed algorithmic
solution to rectify this via dynamic consensus that can be found
in~\cite{AC-JC:16-auto} which could be modified for a Newton-like
method. Given these assumptions, $\Pc 1$ is equivalent to:
\begin{subequations} \label{eq:p2-opt}
	\begin{align}
	\Pc 2: \
	& \underset{z}{\text{min}}
	& & f(x^0 + Lz) = \sum_{i=1}^n f_i(x^0_i + L_i z) \label{eq:p2-cost}\\
	& \text{subject to}
	& & \underline{x} - x^0 - Lz \preceq \zeros[n], \label{eq:p2-box-const1} \\
	&&& x^0 + Lz -\overline{x} \preceq \zeros[n] . \label{eq:p2-box-const2}
	\end{align}
\end{subequations}
Using the property that $\ones[n]$ is an eigenvector of $L$
associated with the eigenvalue $0$, we have that $\ones[n]^\top
(x^0 + Lz) = d$.  Newton descent for centralized solvers is given
in~\cite{SB-LV:04}; in our distributed framework, the row space of the
Laplacian is a useful property to
address~\eqref{eq:p1-resource-const}.

\begin{rem}\longthmtitle{Relaxing Assumption~\ref{ass:initial}}\label{rem:robust}
  The assumption on the initial condition can render the formulation
  vulnerable to implementation errors and cannot easily accommodate
  packet drops in a distributed algorithm. A potential workaround for
  this is outlined here. Consider, instead
  of~\eqref{eq:p1-resource-const} in $\Pc 1$, the $n$ linear
  constraints:
	\begin{equation}\label{eq:dist-equal}
	x + Lz = \overline{d},
	\end{equation}
	where $\overline{d}\in\real^n , \ones[n]^{\top} \overline{d} =
        d$ and~\eqref{eq:p1-resource-const} can be recovered by
        multiplying~\eqref{eq:dist-equal} from the left by
        $\ones[n]$. (As an aside, it may be desirable to impose
        sparsity on $\overline{d}$ so that only some agents need
        access to global problem data). Both $x\in\real^n$ and
        $z\in\real^n$ become decision variables, and agent $i$ can
        verify the $\supscr{i}{th}$ component of~\eqref{eq:dist-equal}
        with one-hop neighbor information. Further, a distributed
        saddle-point algorithm can be obtained by assigning a dual
        variable to~\eqref{eq:dist-equal} and proceeding as
        in~\cite{AC-EM-SHL-JC:18-tac}.
	
	We provide a simulation justification for this approach in
        Section~\ref{ssec:sims-robust-dana}, although the analysis of
        robustness to perturbations and packet drops is ongoing and
        outside the scope of this paper. For now we strictly impose
        Assumption~\ref{ass:initial}. 
\end{rem}

We aim to leverage the freedom given by the elements of $L$ in order
to compute an approximate Newton direction to $\Pc 2$. To this end, we
adopt the following assumption.
\begin{assump}\longthmtitle{Cost Functions}\label{ass:cost}
	The local costs $f_i$ are twice continuously differentiable and
	strongly convex with bounded second-derivatives given by
	\begin{equation*}
	0 < \delta_i \leq \dfrac{\partial^2 f_i}{\partial x_i^2} \leq \Delta_i,
	\end{equation*}
	for every $i\in\nodes$ with given $\delta_i, \Delta_i\in\real_+$.
\end{assump}
This assumption is common in other distributed Newton or Newton-like
methods, e.g.~\cite{AM-QL-AR:17,DJ-JX-JM-14} and in classical convex
optimization~\cite{SB-LV:04,YN:13}. Assumption~\ref{ass:cost} is
necessary to attain convergence in our computation of the Newton
step/direction and to construct the notion of an optimal edge
weighting $L$.
We adopt the shorthands $H(x) := \nabla_{xx}
f(x)$,
$H_\delta := \diag{\delta}$, and $H_\Delta := \diag{\Delta}$ as the
diagonal matrices with elements given by $\partial^2
	f_i(x_i)/ \partial x_i^2$, $\delta_i$, and $\Delta_i$,
respectively. 

Next, for the purpose of developing a distributed Newton-like method,
we must slightly rethink the idea of inverting a Hessian matrix. By
application of the chain rule, we have that $\nabla_{zz}f(x^0 + Lz) =
LH(x^0 + Lz)L$. Clearly, $\nabla_{zz}f$ is non-invertible due to the
smallest eigenvalue of $L$ fixed at zero, a manifestation of the
equality constraint in the original problem $\Pc 1$. We instead focus
on the $n-1$ nonfixed eigenvalues of $\nabla_{zz}f$ to employ the Taylor
expansion outlined in Section~\ref{ssec:Taylor}. To this end,
we project $LH(x^0 + Lz)L$ to the $\real^{(n-1)\times (n-1)}$ space with a
coordinate transformation; the justification for this and relation to
the traditional Newton method are made explicitly clear in
Section~\ref{sec:discrete-red}. We seek a matrix $T\in\real^{n\times n}$ satisfying $T^\top T = I_n - \ones[n]\ones[n]^\top /n$~\cite{SF-AI-LS:03}; the particular matrix $T$ we employ is given as
\begin{equation*}
T\hspace{-1mm} = \hspace{-1mm}\begin{bmatrix}
n\hspace{-1mm} -\hspace{-1mm} 1 + \hspace{-1mm}\sqrt{n} 
& -1 & \cdots & -1 & \dfrac{1}{\sqrt{n}} \\
-1 & \ddots & \cdots & \vdots & \\
\vdots & & \ddots & -1 & \vdots \\
-1 & \cdots & -1 & n\hspace{-1mm} -\hspace{-1mm} 1 + \hspace{-1mm}\sqrt{n} & \\
-1 - \sqrt{n} & \cdots & \cdots & -1 - \sqrt{n} & \dfrac{1}{\sqrt{n}}
\end{bmatrix}\hspace{-1mm}\diag{\hspace{-.5mm}\begin{bmatrix}
	\rho \\ 1
	\end{bmatrix}\hspace{-.5mm}},
\end{equation*}
where $\rho = \sqrt{n(n+1+2\sqrt{n})}^{-1}\ones[n-1]$. This choice of $T$ has the effect of projecting the null-space of the Hessian onto the $\supscr{n}{th}$ row and $\supscr{n}{th}$ column, which is demonstrated by defining
$M(x):= JT^\top LH(x)L TJ^\top \in \real^{(n-1)\times (n-1)}$, where $J
= \begin{bmatrix} I_{n-1} & \zeros[n-1]\end{bmatrix}$. The matrix
$M(x)$ shares its $n-1$ eigenvalues with the $n-1$ nonzero eigenvalues
of $LH(x)L$ at each $x$, and $M(x)^{-1}$ is well defined, which provides
us with a concrete notion of an inverse Hessian. We now adopt the
following assumption.
\begin{assump}\longthmtitle{Convergent Eigenvalues}\label{ass:e-vals}
  For any $x$, the eigenvalues of $I_{n-1} - M(x)$, corresponding to
  the $n-1$ smallest eigenvalues of $I_{n} - LH(x)L$, are contained in
  the unit ball, i.e. $\exists \ \varepsilon < 1$ such
  that
	\begin{equation*}
	-\varepsilon I_{n-1} \preceq I_{n-1} - M(x) \preceq \varepsilon I_{n-1}.
	\end{equation*}
\end{assump}
Technically speaking, we are only concerned with arguments of $M$
belonging to the $n-1$ dimensional hyperplane $\setdef{x^0 +
	Lz}{z\in\real^n}$, although we consider all $x\in\real^n$ for
simplicity. In the following section, we address
Assumption~\ref{ass:e-vals} (Convergent Eigenvalues) 
by minimizing $\epsilon$ via weight design of the Laplacian. By doing
this, we aim to obtain a good approximation of $M^{-1}$ from the
Taylor expansion with small $q$, which lends itself well to the
convergence of the distributed algorithms in
Sections~\ref{sec:discrete-red} and~\ref{sec:cont-time-dana}.
\section{Weight Design of the Laplacian}\label{sec:laplacian-design}
In this section, we pose the nonconvex weight design
problem on the elements of $L$, which formulates as a bilinear optimization to be solved by a central authority.
To make this problem tractable, we develop a convex approximation and
demonstrate that the solution is guaranteed to satisfy
Assumption~\ref{ass:e-vals}. Next, we provide a lower bound on the
solution to the nonconvex
problem. This gives a measure of performance for
evaluating our approximation.
\subsection{Formulation and Convex Approximation}\label{sec:topology-design}
Our approach hearkens to the intuition on the rate of convergence of
the $q$-approximation of $M(x)^{-1}$. We design a weighting scheme for a
communication topology characterized by $L$ which lends itself to a
scalable, fast approximation of a Newton-like direction. To this end,
we minimize $\underset{i,x}{\max} \vert 1 - \mu_i(M(x)) \vert$:
\begin{subequations}		
	\begin{align}
	\Pc 3: \quad
	& \underset{\varepsilon, L}{\text{min}}
	&&\varepsilon \\
	& \text{s.t.} 
	&&\hspace*{-0.3cm}-\varepsilon I_{n-1} \preceq I_{n-1} - M(x) \preceq
	\varepsilon I_{n-1}, \forall x,\label{eq:eval-prob-BMI} \\
	&&& L \ones[n] = \zeros[n], \ L \succeq 0, \ L = L^\top, \\
	&&& L_{ij} \leq 0, j\in \nodes_i, \ L_{ij} = 0,\, j \notin \nodes_i.
	\end{align}
\end{subequations}
Naturally, $\Pc 3$ must be solved offline by a central authority because it requires complete information about the local Hessians embedded in $M(x)$, in addition to being a semidefinite program for which distributed solvers are not mature. Even for a centralized solver $\Pc 3$ is hard for a
few reasons, the first being that~\eqref{eq:eval-prob-BMI}
is a function over all possible $x \in \real^n$. To reconcile with
this, we invoke Assumption~\ref{ass:cost} on the cost functions and
write $M_\delta = JT^\top LH_\delta L TJ^\top$ and $M_\Delta = JT^\top
LH_\Delta L TJ^\top$. Then,~\eqref{eq:eval-prob-BMI} is equivalent to
\begin{subequations}
	\begin{align}
	-(\varepsilon_-  + 1)I_{n-1} + M_\delta &\preceq 0, \label{eq:BMI-epminus} \\
	(1 - \varepsilon_+)I_{n-1} - M_\Delta &\preceq 0, \label{eq:BMI-epplus} \\
	\varepsilon_- = \varepsilon_+, \label{eq:BMI-equal}
	\end{align}
\end{subequations}
where the purpose of introducing $\varepsilon_-$ and $\varepsilon_+$
will become clear in the discussion that follows.

The other difficult element of $\Pc 3$ is the nonconvexity stemming
from~\eqref{eq:BMI-epminus}--\eqref{eq:BMI-epplus} being bilinear in
$L$. 
There are path-following techniques available to solve bilinear
problems of this form~\cite{AH-JH-SB:99}, but simulation results do
not produce satisfactory solutions for problems of the form $\Pc
3$. Instead, we aim to develop a convex approximation of $\Pc 3$ which
exploits its structure. Consider~\eqref{eq:BMI-epminus}
and~\eqref{eq:BMI-epplus} separately by
relaxing~\eqref{eq:BMI-equal}. In fact,~\eqref{eq:BMI-epminus} may be
rewritten in a convex manner. To do this, write $L$ as a weighted
product of its incidence matrix, $L = E^\top X E$. Applying
Lemma~\ref{lem:schur-comp} makes the constraint become
\begin{equation}\label{ineq:convex_constraint}
\begin{bmatrix}
(\varepsilon_- + 1)I_{n-1} & JT^\top E^\top X E 			\\
E^\top X E T J^\top & H_\delta^{-1}
\end{bmatrix} \succeq 0.
\end{equation}
As for~\eqref{eq:BMI-epplus}, consider the approximation $LH_\Delta L
\approx \left(\dfrac{\sqrt{H_\Delta }L + L\sqrt{H_\Delta
}}{2}\right)^2$. This approximation can be thought of as
a rough completion of squares, which lends itself well to our approach of convexifying~\eqref{eq:BMI-epplus}. One should not expect the approximation to be reliably ``better" or ``worse" than the BMI; rather, it is only intended to reflect the original constraint more than a simple linearization. To this end,
substitute this in $M_\Delta$ to get
\begin{equation*}
\begin{aligned}
&\dfrac{1}{4}J T^\top (\sqrt{H_\Delta }L + L\sqrt{H_\Delta})^2
T J^\top \succeq (1 - \varepsilon_+)I_{n-1} \\
&\dfrac{1}{2}J T^\top (\sqrt{H_\Delta }L + L\sqrt{H_\Delta}) T
J^\top
\succeq \sqrt{(1 - \varepsilon_+)}I_{n-1} \\
&\dfrac{1}{2}J T^\top
(\sqrt{H_\Delta }L + L\sqrt{H_\Delta }) T J^\top \hspace{-1mm}\succeq \\
&\hspace{35mm}(1\hspace{-.5mm} - \hspace{-.5mm}
\dfrac{\varepsilon_+}{2}\hspace{-.5mm}
+\hspace{-.5mm}\dfrac{\varepsilon_+^2}{8} +
O(\varepsilon_+^3))I_{n-1},
\end{aligned}
\end{equation*}
where the second line uses the property that $T J^\top J T^\top = I_n
- \ones[n]\ones[n]^\top / n$ is idempotent and that
\begin{equation*}
\begin{aligned}
\left(\dfrac{1}{2}J T^\top (\sqrt{H_\Delta}L + L\sqrt{H_\Delta}) T
J^\top\right)^2 &\succeq (1 - \varepsilon_+)I_{n-1} \\
	&\qquad\qquad\succeq 0	\\
\Leftrightarrow \dfrac{1}{2}J T^\top (\sqrt{H_\Delta}L +
L\sqrt{H_\Delta}) T J^\top &\succeq \sqrt{1-\varepsilon_+}I_{n-1}
\succeq 0,
\end{aligned}
\end{equation*}
see~\cite{JV-RB:00}. The third line expresses the right-hand
side as a Taylor expansion about $\varepsilon_+ = 0$. Neglecting the
higher order terms $O(\varepsilon_+^3)$ and applying
Lemma~\ref{lem:schur-comp} gives
\begin{equation}\label{ineq:nonconvex_constraint}
\begin{bmatrix}
\bullet & \dfrac{1}{\sqrt{8}}\varepsilon_+ I_{n-1} 			\\
\dfrac{1}{\sqrt{8}}\varepsilon_+ I_{n-1} & I_{n-1}
\end{bmatrix} \succeq 0,
\end{equation}
with ${\bullet = \dfrac{1}{2}J T^\top (\sqrt{H_\Delta}L +
	L\sqrt{H_\Delta}) T J^\top} - {(1 -
	\dfrac{1}{2}\varepsilon_+)I_{n-1}}$.

Returning to $\Pc 3$, note that the latter three constraints are
satisfied by $L = E^\top X E$. Then, the approximate reformulation of
$\Pc 3$ can be written as
\begin{equation*}		
\begin{aligned}
\Pc 4: \quad &\underset{\varepsilon_-, \varepsilon_+,
	X}{\text{min}}
&&\max(\varepsilon_-, \varepsilon_+) \\
& \text{s.t.}
&& \varepsilon_- \geq 0, \varepsilon_+ \geq 0,			\\
&&& X \succeq 0, \eqref{ineq:convex_constraint}, \eqref{ineq:nonconvex_constraint}.	\\
\end{aligned}
\end{equation*}
This is a convex problem in $X$ and solvable in polynomial time. To
improve the solution, we perform some post-scaling. Take $L^\star_0 =
E^\top X_0^\star E$, where $X_0^\star$ is the solution to $\Pc 4$, and
let $M_{\Delta 0}^\star = JT^\top L^\star_0H_\Delta L^\star_0TJ^\top,
M_{\delta 0}^\star = JT^\top L^\star_0H_\delta
L^\star_0TJ^\top$. Then, consider
\begin{equation*}
\beta = \sqrt{\dfrac{2}{\mu_1(M_{\delta 0}^\star) + \ \mu_{n-1}(M_{\Delta 0}^\star)}},
\end{equation*}
and take $L^\star = \beta L^\star_0$. This shifts the eigenvalues of
$M^\star_0(x)$ to $M^\star(x)$ (defined similarly via $L^\star$) such
that $1 - \mu_1(M_\delta^\star) = -(1 - \mu_{n-1}(M_\Delta^\star))$,
which shrinks $\underset{i,x}{\max}(\vert 1 -
\mu_i(M^\star(x))\vert)$. We refer to this metric as
$\varepsilon_{L^\star} := \underset{i,x}{\max}(\vert 1 -
\mu_i(M^\star(x))\vert)$, and it can be verified that this
post-scaling satisfies
Assumption~\ref{ass:e-vals} with regard to $\varepsilon_{L^\star}$. To see this, first consider scaling $L$ by an arbitrarily small constant, which places the eigenvalues of $I_{n-1} - M(x)$ very close to $1$ and satisfies Assumption~\ref{ass:e-vals}. Then, consider gradually increasing this constant until the lower bound on the minimum eigenvalue and upper bound on the maximum eigenvalue of $I_{n-1} - M(x)$ are equal in magnitude. This is precisely the scaling produced by $\beta$. Then, the solution to $\Pc 4$ followed by
a post scaling by $\beta$ given by $L^\star$ is an approximation of
the solution to the nonconvex problem $\Pc 3$ with the sparsity
structure preserved.

	\begin{rem}\longthmtitle{Unknown Local Hessian Bounds}\label{rem:glob-hess-bound}
		It may be the case that a central entity tasked with computing some $L^\star$ does not have access to the local bounds $\delta_i, \Delta_i, \forall i$. In this case, globally known bounds $\delta \leq \delta_i, \Delta_i \leq \Delta, \forall i$ can be substituted in place of the local values in the formulation of $\Pc 4$. It can be verified that this will result in a more conservative formulation, and that the resulting $L^\star$ will still satisfy Assumption~\ref{ass:e-vals} at the expense of possibly larger $\varepsilon$.
	\end{rem}

\subsection{A Bound on Performance}
We are motivated to find a
``best-case scenario'' 
for our solution given the structural constraints of the
network. Instead of solving
$\Pc 3$ for $L$, we solve it for some $A$ where $A_{ij} = 0$ for
$j\notin \nodes_i^2$, i.e. the two-hop neighbor structure of the
network and sparsity structure of $LH(x)L$. Define $M_A := JT^\top A T
J^\top$. This problem is:
\begin{equation*}\label{eq:lower bound}
\begin{aligned}
\Pc 5: \quad &\underset{\varepsilon, A}{\text{min}}
&&\varepsilon \\
& \text{s.t.}
&&- \varepsilon I_{n-1} \preceq I_{n-1} - M_A \preceq \varepsilon I_{n-1}, \\
&&& A \ones[n] = \zeros[n], \ A \succeq 0, \\
&&& A_{ij} = 0, j \notin \nodes^2_i.
\end{aligned}
\end{equation*}
This problem is convex in $A$ and produces a solution $\varepsilon_A$,
which serves as a lower bound for the solution to $\Pc 3$. 
It should not be expected that this lower bound is tight or achievable by ``reverse engineering" an $L^\star$ with the desired sparsity from the solution $A^\star$ to $\Pc 5$, rather, $\varepsilon_{L^\star} - \varepsilon_A$ gives just a rough indication of how close
$\varepsilon_{L^\star}$ is to the conservative lower bound of
$\Pc 3$.
\section{Discrete Time Algorithm for Relaxed Economic Dispatch}
\label{sec:discrete-red}
In this section, we focus on a \emph{relaxed}
version of $\Pc 2$ to develop a direct relation between traditional
discrete-time Newton descent and our distributed, approximate
method. First, we state the relaxed problem and define the approximate
Newton step. We then state the $\distnewtondisc$ algorithm 
and provide a rigorous study of its convergence properties.
\subsection{Characterization of the Approximate Newton
	Step}\label{ssec:char_approx_newton}
Even the traditional centralized Newton method is not well-suited to
solve $\Pc 1$ due to the box constraints~\eqref{eq:p1-box-const}. For
this reason, for now we focus on the relaxed problem
\begin{subequations} \label{eq:p6-opt}
	\begin{align}
	\Pc 6: \
	& \underset{x}{\text{min}}
	&& \hspace{-12mm} f(x) = \sum_{i=1}^n f_i(x_i), \label{eq:p6-cost}\\
	& \text{subject to}
	&& \hspace{-12mm} \sum_{i=1}^n x_i = d. \label{eq:p6-resource-const}
	\end{align}
\end{subequations}
The equivalent unconstrained problem in $z$ is
\begin{equation} \label{eq:p7-opt} \Pc 7: \
\underset{z}{\text{min}} \quad g(z) := f(x^0 + Lz) =
\sum_{i=1}^n f_i(x^0_i + L_i z).
\end{equation}

\begin{rem}\longthmtitle{Nonuniqueness of Solution}
	Given a $z^\star$ which
	solves $\Pc 7$, the set of solutions can be characterized by
	$\setdef{z^{\star\prime}}{z^{\star\prime} = z^\star + \gamma
		\ones[n]\, \ \gamma\in\real}$. The fact that $z^{\star\prime}$ is
	a solution is due to $\nulo{(L)} = \spn{(\ones[n])}$, and the fact
	that this characterizes the entire set of solutions is due to
	$\nulo{(\nabla_{zz}g(z))} = \spn{(\ones[n])}$.
\end{rem}
To solve $\Pc 6,$ we aim to implement a descent method in $x$ via the
dynamics
\begin{equation} \label{eq:short algorithm} x^{+} = x +
\alpha L \subscr{\tilde{z}}{nt},
\end{equation}
where $\subscr{\tilde{z}}{nt}$ is the \emph{approximate} Newton step that we seek to compute distributively, and $\alpha > 0$ is a
fixed step size.

It is true that $\Pc 7$ is unconstrained with respect to $z$, although
we have already alluded to the fact that the Hessian matrix
$\nabla_{zz} g(z) = LH(x+Lz)L$ is rank-deficient stemming
from~\eqref{eq:p6-resource-const}. We now reconcile this by deriving a
well defined Newton step in a reduced variable
$\hat{z}\in\real^{n-1}$.  Consider a change of coordinates by the
orthogonal matrix $T$ defined in Section~\ref{sec:prob-statement} and
write $z=TJ^\top\hat{z}$. Taking the gradient and Hessian of $g(z)$
with respect to $\hat{z}$ gives
\begin{equation*}
\begin{aligned}
\nabla_{\hat{z}} g(z) &=
JT^\top \nabla_z g(z) = JT^\top L \nabla_x f(x+LTJ^\top \hat{z}) \\
\nabla_{\hat{z}\hat{z}} g(z) &= JT^\top LH(x + LTJ^\top \hat{z})LTJ^\top \\
&= M(x + LTJ^\top \hat{z}).
\end{aligned}
\end{equation*}
Notice that the zero eigenvalue of $\nabla_{zz} g(z)$ is eliminated by
this projection and the other eigenvalues are preserved. Evaluating at
$x + LTJ^\top \hat{z}\big|_{\hat{z}=0}$, the Newton step in $\hat{z}$
is now well defined as $\subscr{\hat{z}}{nt}
:=-\nabla_{\hat{z}\hat{z}} g(0)^{-1} \nabla_{\hat{z}} g(0) =
-M(x)^{-1}JT^\top L \nabla_x f(x)$.

Consider now a $q$-approximation of $M(x)^{-1}$ given by
$\sum_{p=0}^q (I_{n-1} - M(x))^p$ and return to the original
coordinates to obtain the \emph{approximate} Newton direction
$L\tilde{z}_{nt}$:
\begin{equation*}
L\tilde{z}_{nt} = -LTJ^\top \sum_{p=0}^{q} (I_{n-1} - M(x))^p JT^\top L \nabla_x f(x).
\end{equation*}
With the property that $LTJ^\top JT^\top L = L^2$, rewrite
$L\subscr{\tilde{z}}{nt}$:
\begin{equation}
L\subscr{\tilde{z}}{nt} = -L \sum_{p=0}^q (I_n - LH(x)L)^p L \nabla_x f(x).
\label{eq:znewton-step}
\end{equation}
It can be seen via eigendecomposition of $I_n - LHL$, which is normal, and application
of Assumption~\ref{ass:e-vals} that the terms $L(I_n - LHL)^p$ become
small with $p\rightarrow\infty$ at a rate dictated by
$\varepsilon$. Note that there is a nonconverging mode
  of the sum corresponding to the eigenspace spanned by $\ones[n]$,
  but this is mapped to zero by left multiplication by $L$.
This expression can be computed distributively: each
multiplication by $L$ encodes a communication with the neighbor set of
each agent, and we utilize recursion to perform the computation
efficiently, which is formally described in
Algorithm~\ref{alg:approx-newton}.
\subsection{The $\distnewton$ Algorithm}
We now have the tools to introduce the \\
$\distnewtondisc$ algorithm, or
$\danad$.

\begin{algorithm}
	\caption{$\danad_i$}
	\label{alg:approx-newton}
		\begin{algorithmic}[1]
		\Require $L_{ij}$ for $j \in \{i\} \cup \nodes_i$ and
		communication with nodes $j \in \nodes_i \cup \nodes_i^2$
		\Procedure{Newton$_i$}{$x_i^0,L_i,f_i,q$}
		\State Initialize $x_i\gets x_i^0$
		\Loop 
		\State Compute 
		$\dfrac{\partial f_i}{\partial x_i}$, $\dfrac{\partial^2
			f_i}{\partial x_i^2}$; send to $j\in\nodes_i$, $\nodes^2_i$ \vspace{1mm} \label{alg1:loop first term}
		\State $y_i \gets L_{ii}\dfrac{\partial f_i}{\partial x_i} +
		\sum_{j \in\nodes_i} L_{ij} \dfrac{\partial f_j}{\partial x_j}$
		\State $z_i \gets -y_i$
		\State $p_i \gets 1$
		\While{$p_i \leq q$}		\label{alg1:loop pth term}
		\State 	Acquire $y_j$ from $j \in \nodes_i^{2}$\label{alg1:y-info}
		\State 	$w_i = (I_n - LH(x)L)_i y$\label{alg1:w-comp}
		\State 	$y_i \gets w_i$
		\State 	$z_i \gets z_i - y_i$	\label{alg1:sum z}
		\State $p_i \gets p_i + 1$
		\EndWhile				
		\State Acquire $z_j$ for $j \in \nodes_i$ \label{alg1:loop step}
		\State $x_i \gets x_i + \alpha\left( L_{ii}z_i + \sum_{j \in\nodes_i}L_{ij}z_j\right)$
		\EndLoop
		\State \textbf{return} $x_i$
		\EndProcedure
	\end{algorithmic}
\end{algorithm}

The algorithm is constructed directly from~\eqref{eq:short algorithm}
and~\eqref{eq:znewton-step}. The $L\nabla_xf(x^k)$ factor
of~\eqref{eq:znewton-step} is computed first in the loop starting on
line~\ref{alg1:loop first term}. Then, each additional term of the sum
is computed recursively in the loop starting on line~\ref{alg1:loop
  pth term}, where $y$ implicitly embeds the
  exponentiation by $p$ indicated in~\eqref{eq:znewton-step}, $z$
  accumulates each term of the summation of~\eqref{eq:znewton-step}, $w$ is used as an intermediate variable, and $p_i$ is used as a simple counter. We
  introduce some abuse of notation by switching to vector and matrix
  representations of local variables in line~\ref{alg1:w-comp}; this
  is done for compactness and to avoid undue clutter. Note that the
  diagonal elements of $H(x)$ are given by $\partial^2 f_j / \partial
  x_j^2$ and the matrix and vector operations can be implemented
  locally for each agent using the corresponding elements $y_j$,
  $L_{ij}$, and $L^2_{ij}$. The one-hop and two-hop communications of
  the algorithm are contained in lines~\ref{alg1:loop first term}
  and~\ref{alg1:y-info}, where line~\ref{alg1:loop first term} calls
  upon local evaluations of the gradient and Hessian. (In principle,
  Hessian information could be acquired along with $y_j$ in the first
  iteration of the inner loop to utilize one fewer two-hop
  communication, but it need only be acquired once per outer loop.)
  The information is utilized in local computations indicated the next
  line in each case. It is understood that agents perform
  communications and computations synchronously.

The outer loop of
the algorithm corresponding to~\eqref{eq:short
    algorithm} is performed starting on line~\ref{alg1:loop step}. If
only one-hop communications are available, each outer loop of the
algorithm requires $2q+1$ communications. The process repeats until
desired accuracy is achieved. If $q$ is increased, it requires
additional communications, but the step approximation gains accuracy.
\subsection{Convergence Analysis}~\label{ssec:disc-conv-anl}
This section establishes convergence properties of the $\danad$
algorithm for problems of the form $\Pc 6$. For the sake of cleaner
analysis, we will reframe the algorithm as solving $\Pc 7$ via
\begin{equation} \label{eq:short algorithm-z} z^{+} = z -
\alpha A_q(z) \nabla_z g(z),
\end{equation}
where $A_q(z):= \sum_{p=0}^q (I_n - LH(x^0+Lz)L)^p$. Then, note that
the solution $z^\star$ to $\Pc 7$ solves $\Pc 6$ by $x^\star = x^0 +
Lz^\star$ and that~\eqref{eq:short algorithm-z} is equivalent
to~\eqref{eq:short algorithm}-\eqref{eq:znewton-step} and Algorithm~\ref{alg:approx-newton}.
\begin{rem}\longthmtitle{Initial Condition, 
		Trajectories, \& Solution}\label{rem:init-traj-sol}
	Consider an initial condition $z(0)\in\real^n$ with $\ones[n]^\top
	z(0) = \omega$. Due to $A_q(z)\nabla_z g(z)\perp\ones[n]$, the trajectories under~\eqref{eq:short algorithm-z} are contained
	in the set $\setdef{z}{z = \tilde{z} + (\omega/n) \ones[n], \
		\tilde{z} \perp \ones[n]}$. The solution $x^\star = x^0 +
	Lz^\star$ to $\Pc 6$ is agnostic to $(\omega/n) \ones[n]$ due to
	$\nulo{(L)} = \spn{(\ones[n])}$, so we consider the solution $z^\star$ uniquely satisfying $\ones[n]^\top z^\star = \omega$.
\end{rem}
\begin{thrm}\longthmtitle{Convergence of  $\danad$}\label{thm:cvg-dana-d}
	{\rm Given an initial condition $z(0)\in\real^n$, if
		Assumption~\ref{ass:conn-graph}, on the bidirectional connected
		graph, Assumption~\ref{ass:initial}, on the feasibility of the
		initial condition, Assumption~\ref{ass:cost}, on bounded Hessians,
		and Assumption~\ref{ass:e-vals}, on convergent eigenvalues, hold,
		then the $\danad$ dynamics~\eqref{eq:short algorithm-z} converge
		asymptotically to an optimal solution $z^\star$ of $\Pc 7$
		uniquely satisfying $\ones[n]^\top z^\star = \ones[n]^\top z(0)$
		for any $q \in \natural$ and $\alpha <
		\dfrac{2(1-\epsilon)}{(n-1)(1+\epsilon)(1-\epsilon^{q+1})}$.  }
\end{thrm}

\begin{proof} Consider the discrete-time Lyapunov function
	\begin{equation*}
	V(z) = g(z) - g(z^\star)
	\end{equation*}
	defined on the domain $\dom{(V)} = \setdef{z}{\ones[n]^\top z =
		\ones[n]^\top z(0)}$. From the theorem statement and in
	consideration of Remark~\ref{rem:init-traj-sol}, the trajectories of
	$z$ under~\eqref{eq:short algorithm-z} are contained in the domain
	of $V$, and $V(z) > 0, \forall z\in\dom{(V)}, z\neq z^\star$.  To
	prove convergence to $z^\star$, we must show negativity of
	\begin{equation}\label{eq:disc-lyap-diff}
	V(z^+) - V(z) = g(z^+) - g(z).
	\end{equation}
	From the weight design of $L$ (Assumption~\ref{ass:e-vals}),
	we have $\nabla_{zz}g(z) \preceq (1+\epsilon)I_n$,
	$\epsilon\in[0,1)$. This implies
	\begin{equation*}
	\begin{aligned}
	g(z^+) &= g(z) + \nabla_z g(z)^\top (z^+ - z) \\
	& \qquad + \dfrac{1}{2}(z^+ - z)^\top \nabla_{zz}g(z^\prime)(z^+ - z) \\
	&\leq g(z) + \nabla_z g(z)^\top (z^+ - z) + 
	\dfrac{1+\epsilon}{2}\| z^+ - z\|^2_2,
	\end{aligned}
	\end{equation*}
	which employs the standard quadratic expansion of convex
	functions via some $z^\prime$ in the segment
	extending from $z$ to $z^+$ (see e.g. $\S 9.1.2$ of~\cite{SB-LV:04}).
	Substituting~\eqref{eq:short algorithm-z} gives
	\begin{equation}~\label{eq:gz-ineq}
	\begin{aligned}
	g(z^+) &\leq g(z) - \alpha \nabla_z g(z)^\top A_q(z) \nabla_z g(z) \\
	& \qquad+ \dfrac{(1+\epsilon)\alpha^2}{2}\| A_q(z) \nabla_z g(z)\|^2_2.
	\end{aligned}
	\end{equation}
	We now show $A_q(z) \succ 0$ by computing its eigenvalues. Note
	$\mu_i(I_n -
	LH(x^0+Lz)L)\in[-\varepsilon,\varepsilon]\cup\{1\}$. Let $\mu_i(I_n
	- LH(x^0+Lz)L) = \eta_i(z)$ for $i\in\until{n-1}$. The terms of
	$A_q(z)$ commute and it is normal, so it can be diagonalized as
	\begin{equation*}
	\begin{aligned}
	A_q(z) = W(z) \begin{bmatrix}
	\ddots & & & \\
	& \dfrac{1 - \eta_i(z)^{q+1}}{1 - \eta_i(z)} & & \\
	& & \ddots \\
	& & & q+1
	\end{bmatrix} W(z)^\top,
	\end{aligned}
	\end{equation*}
	where the columns of $W(z)$ are the eigenvectors of $A_q(z) \succ
	0$, the last column being $\ones[n]$, and the terms of the diagonal
	matrix are its eigenvalues computed by a geometric series.
	
	For now, we only use the fact that $A_q(z)\succ 0$ to justify the
	existence of $A_q(z)^{1/2}$. Returning to~\eqref{eq:gz-ineq},
	\begin{equation}\label{eq:alpha-bound-0}
	\begin{aligned}
	g(z^+) &\leq g(z) - \alpha \biggl(\| A_q(z)^{1/2} \nabla_z g(z)\|^2_2 \\
	& \qquad - \dfrac{(1+\epsilon)\alpha}{2}\| A_q(z) \nabla_z g(z)\|^2_2\biggr).
	\end{aligned}
	\end{equation}
	Recall $\nabla_z g(z)\perp\ones[n]$ and that $\ones[n]$ is an
	eigenvector of $A_q(z)$ associated with the eigenvalue
	$q+1$. Consider a matrix $\widetilde{A}_q(z)$ whose rows are
	projected onto the subspace spanning the orthogonal complement of
	$\ones[n]$. More precisely, writing $\widetilde{A}_q(z)$ via its
	diagonalization gives
	\begin{subequations}\label{eq:tilde-Aq}
		\begin{align}
		&\widetilde{A}_q(z) = W(z) \begin{bmatrix}
		\ddots & & & \\
		& \dfrac{1 - \eta_i(z)^{q+1}}{1 - \eta_i(z)} & & \\
		& & \ddots \\
		& & & 0
		\end{bmatrix} W(z)^\top,\label{eq:tilde-Aq-def}	\\
		&\widetilde{A}_q(z) \nabla_z g(z) = A_q(z) \nabla_z g(z),
		\label{eq:tilde-Aq-lin} \\
		&\widetilde{A}_q(z)^{1/2} \nabla_z g(z) = A_q(z)^{1/2} \nabla_z
		g(z).
		\label{eq:tilde-Aq-sqrt}
		\end{align}
	\end{subequations}
	Combining~\eqref{eq:alpha-bound-0}--\eqref{eq:tilde-Aq} gives the
	sufficient condition on $\alpha$:
	\begin{equation}\label{eq:alpha-bound-2}
	\alpha < \dfrac{2\| \widetilde{A}_q(z)^{1/2} 
		\nabla_z g(z)\|^2_2}{(1+\epsilon)\| \widetilde{A}_q(z) 
		\nabla_z g(z)\|^2_2}.
	\end{equation}
	Multiply the top and bottom of the righthand side
	of~\eqref{eq:alpha-bound-2} by $\| \widetilde{A}_q(z)^{1/2}\|
	^2_2$ and apply submultiplicativity of $\| \cdot \|^2_2$:
	\begin{equation}\label{eq:Aq-ineq}
	\dfrac{2}{(1+\epsilon)\| \widetilde{A}_q(z)^{1/2}\|^2_2} 
	\leq \dfrac{2\| \widetilde{A}_q(z)^{1/2} \nabla_z g(z)
		\|^2_2}{(1+\epsilon)\| \widetilde{A}_q(z) \nabla_z g(z)\|^2_2}.
	\end{equation}
	Finally, we bound the lefthand side of~\eqref{eq:Aq-ineq} from below
	by substituting $\eta_i(z)$ with $\epsilon$:
	\begin{equation}\label{eq:bound-Aq-half}
	\begin{aligned}
	\| \widetilde{A}_q(z)^{1/2}\|^2_2
	&= \sum_i^{n-1} \dfrac{1-\eta_i(z)^{q+1}}{1-\eta_i(z)} \\
	&\leq (n-1)\dfrac{1-\epsilon^{q+1}}{1-\epsilon}, \quad \forall
	z\in\real^n.
	\end{aligned}
	\end{equation}
	Combining~\eqref{eq:bound-Aq-half} with~\eqref{eq:Aq-ineq} gives the
	condition on $\alpha$ in the theorem statement and completes the
	proof.		
\end{proof}
In practice, we find this to be a very conservative bound on $\alpha$
due to the employment of many inequalities which simplify the
analysis. We note that designing $L$ effectively such that $\epsilon$
is close to zero allows for more flexibility in choosing $\alpha$
large, which intuitively indicates the Taylor approximation of the
Hessian inverse converging with greater accuracy in fewer terms $q$.
\begin{thrm}\longthmtitle{Linear Convergence of  $\danad$}\label{thm:cvg-dana-d-exp}
	{\rm Given an initial condition $z(0)\in\real^n$ and step size
		$\alpha =
		\dfrac{(1-\epsilon)}{(n-1)(1+\epsilon)(1-\epsilon^{q+1})}$, if
		Assumption~\ref{ass:conn-graph}, on the bidirectional connected
		graph, Assumption~\ref{ass:initial}, on the feasibility of the
		initial condition, Assumption~\ref{ass:cost}, on bounded Hessians,
		and Assumption~\ref{ass:e-vals}, on convergent eigenvalues, hold,
		the $\danad$ dynamics~\eqref{eq:short algorithm-z} converge
		\emph{linearly} to an optimal solution $z^\star$ of $\Pc 7$
		uniquely satisfying $\ones[n]^\top z^\star = \ones[n]^\top z(0)$
		in the sense that $g(z^+) - g(z) \leq
		-\dfrac{(1-\varepsilon)^4(1+\varepsilon(-\varepsilon)^q)^2\|
			z-z^\star\|_2^2}{2(n-1)^2(1+\varepsilon)^3(1-\varepsilon^{2(q+1)})}$
		for any $q \in \natural$.  }
\end{thrm}
\begin{proof}
	Define
	\begin{equation*}
	\begin{aligned}
	c_1(z) &= \| \widetilde{A}_q(z)^{1/2}\nabla_z g(z)\|_2^2, \\
	c_2(z) &= \dfrac{(1+\varepsilon)}{2}\|
	\widetilde{A}_q(z)\nabla_z g(z)\|_2^2,
	\end{aligned}
	\end{equation*}
	with $\widetilde{A}_q(z)$ defined as
	in~\eqref{eq:tilde-Aq-def}.
	Recalling~\eqref{eq:alpha-bound-0}--\eqref{eq:tilde-Aq}, consider $\bar{\alpha} = 2\alpha$ as the smallest step size such
	that $- \bar{\alpha} c_1(z) + \bar{\alpha}^2c_2(z)$ is not strictly
	negative for all $z$, which is obtained from the result of Theorem~\ref{thm:cvg-dana-d}. 
	Then,
	\begin{equation}\label{eq:c-bounds-1}
	\begin{aligned}
	-\bar{\alpha} c_1(z) + \bar{\alpha}^2c_2(z) &\leq 0 \Rightarrow \\
	-\alpha c_1(z) + \alpha^2 c_2 (z) &\leq -\alpha^2 c_2(z).
	\end{aligned}
	\end{equation}
	The second line is obtained from the first by substituting
	$\bar{\alpha} = 2\alpha$. We now consider an
	implementation of $\danad$ with $\alpha$. From~\eqref{eq:alpha-bound-0}
	and substituting
	via~\eqref{eq:tilde-Aq-lin}--\eqref{eq:tilde-Aq-sqrt}, we obtain
	$g(z^+) - g(z) \leq -\alpha c_1(z) + \alpha^2 c_2 (z)$. Combining
	this with the second line of~\eqref{eq:c-bounds-1},
	\begin{equation}\label{eq:c-bounds-2}
	g(z^+) - g(z) \leq -\alpha^2 c_2 (z).
	\end{equation}
	We seek a lower bound for $\widetilde{A}_q(z)$. Consider its
	definition~\eqref{eq:tilde-Aq-def}, where a lower bound can be
	obtained by substituting each $\eta_i(z)$ by $-\varepsilon$. Then,
	\begin{equation*}
	\widetilde{A}_q(z) \succeq \dfrac{1+\varepsilon 
		(-\varepsilon)^{q}}{1 + \varepsilon} \left(I_n - 
	\dfrac{\ones[n]\ones[n]^\top}{n}\right).
	\end{equation*}
	Returning to~\eqref{eq:c-bounds-2} and applying the definition of $c_2(z)$,
	\begin{equation}\label{eq:c-bounds-3}
	g(z^+) - g(z) \leq -\dfrac{\alpha^2(1+
		\varepsilon(-\varepsilon)^q)^2}{2(1+\varepsilon)}\| \nabla_z g(z)\|_2^2,
	\end{equation}
	due to $\nulo(I_n - \ones[n]\ones[n]^\top/n) = \spn(\ones[n])$ and
	$\nabla_z g(z) \perp \ones[n]$.
	
	Next, we bound $\| \nabla_z g(z)\|_2^2$. Apply the Fundamental
	Theorem of Calculus to compute $\nabla_z g(z)$ via a line integral. Let $z(s) = sz + (1-s)z^\star$. Then,
	\begin{equation}\label{eq:c-bounds-4}
	\begin{aligned}
	\nabla_z g(z) &= \int_0^1 \nabla_{zz} g(z(s))
	(z-z^\star) ds .
	\end{aligned}
	\end{equation}
	Applying Assumption~\ref{ass:e-vals} (convergent eigenvalues) gives a
	lower bound on the Hessian of $g$, implying a lower bound on its
	line integral:
	\begin{equation}\label{eq:c-bounds-5}
	\begin{aligned}
	&\nabla_{zz}g(z)\succeq
	(1-\varepsilon)(I-\ones[n]\ones[n]^\top/n)
	\Rightarrow	\\
	\int_0^1 &\nabla_{zz} g(z(s))ds \succeq
	(1-\varepsilon)(I-\ones[n]\ones[n]^\top/n).
	\end{aligned}
	\end{equation}
	Factoring out $z-z^\star$ from~\eqref{eq:c-bounds-4} and applying
	the second line of~\eqref{eq:c-bounds-5} gives the lower bound
	\begin{equation}\label{eq:c-bounds-6}
	\begin{aligned}
	\| \nabla_z g(z)\|_2^2 \geq (1-\varepsilon)^2\|
	z-z^\star\|_2^2,
	\end{aligned}
	\end{equation}
	due to $\nulo(I_n - \ones[n]\ones[n]^\top/n) = \spn(\ones[n])$ and
	$z-z^\star \perp \ones[n]$. Combining~\eqref{eq:c-bounds-6}
	with~\eqref{eq:c-bounds-3} and substituting $\alpha$:
	\begin{equation*}
	g(z^+) - g(z) \leq -\dfrac{(1-\varepsilon)^4(1+
		\varepsilon(-\varepsilon)^q)^2\|
		z-z^\star\|_2^2}{2(n-1)^2(1+
		\varepsilon)^3(1-\varepsilon^{2(q+1)})}.
	\end{equation*}
\end{proof}
In principle, this result can be extended to any $\alpha$ which is
compliant with Theorem~\ref{thm:cvg-dana-d}; we have chosen this
particular $\alpha$ for simplicity. The methods we employ to arrive at
the results of Theorems~\ref{thm:cvg-dana-d}
and~\ref{thm:cvg-dana-d-exp} are necessarily conservative. However, in
practice, we find that choosing substantially larger $\alpha$
generally converges to the solution faster. Additionally, we find
clear-cut improved convergence properties for larger $q$ (more
accurate step approximation) and smaller $\varepsilon$ (more effective
weight design). Simulations confirm this in Section~\ref{sec:sims-discuss}. 
\section{Continuous Time Distributed Approximate Newton
	Algorithm}\label{sec:cont-time-dana}
In this section, we develop a continuous-time Newton-like algorithm to
distributively solve $\Pc 2$ for quadratic cost functions. Our method
borrows from and expands upon known results of gradient-based
saddle-point dynamics~\cite{AC-EM-SHL-JC:18-tac}. We provide a rigorous
proof of convergence and an interpretation of the convergence result
for various parameters of the proposed algorithm.
\subsection{Formulation of Continuous Time Dynamics}
First, we adopt a stronger version of
Assumption~\ref{ass:cost}:
\begin{assump}\longthmtitle{Quadratic Cost Functions}\label{ass:quad-cost}
	The local costs $f_i$ are strongly convex and quadratic, i.e. they take the form
	\begin{equation*}
	f_i(x_i) = \dfrac{1}{2}a_i x_i^2 + b_i x_i, \quad i\in\until{n}.
	\end{equation*}
\end{assump}
Note that the Hessian of $f$ with respect to $x$ is now constant, so
we omit the arguments of $H$ and $A_q$ for the remainder of this
section. The dynamics we intend to use to solve $\Pc 2$ are
substantially more complex than those for the problem with no box
constraints, which makes this simplification necessary. In fact, the
quadratic model is very commonly used for generator costs in power
grid operation~\cite{AW-BW-GS:12}.

We aim to solve $\Pc 2$ by finding a saddle point of the associated
Lagrangian $\Ls$.  Introduce the dual variable $\lambda\in\real^{2n}$
corresponding to~\eqref{eq:p2-box-const1}--\eqref{eq:p2-box-const2},
and define $P(z)$ as
\begin{equation*}
P(z) = \begin{bmatrix}
\underline{P}(z) \\ \overline{P}(z)
\end{bmatrix} =
\begin{bmatrix}
\underline{x} - x^0 - Lz \\ x^0 + Lz - \overline{x}
\end{bmatrix} \in\real^{2n}.
\end{equation*}
The Lagrangian of $\Pc 2$ is given by
\begin{equation}
\Ls (z,\lambda) = g(z) + \lambda^\top P(z). \label{eq:Lagrangian}
\end{equation}
We aim to design distributed dynamics which converge to a saddle point
$(z^\star ,\lambda^\star )$ of~\eqref{eq:Lagrangian}, which solves
$\Pc 2$. A saddle point has the property
\begin{equation*}
\Ls (z^\star,\lambda) \leq \Ls (z^\star,\lambda^\star) 
\leq \Ls (z,\lambda^\star), \quad \forall z\in\real^n , \lambda\in\real^n_{\geq 0}.
\end{equation*}
To solve this, consider Newton-like descent dynamics in the primal
variable $z$ and gradient ascent dynamics in the dual variable
$\lambda$ (Newton dynamics are not well defined for linear
functions). First, we state some equivalencies:
\begin{equation}\label{eq:equiv}
\begin{aligned}
\nabla_z \Ls (z,\lambda) &= \nabla_z g(z) + \begin{bmatrix}-L &
L\end{bmatrix}\lambda, \\
\nabla_\lambda \Ls (z,\lambda) &= P(z), \\
\nabla_{zz}\Ls (z,\lambda) &= LHL, \\
\nabla_{\lambda\lambda}\Ls (z,\lambda) &= \zeros[2n\times 2n], \\
\nabla_{\lambda z}\Ls (z,\lambda) &= \nabla_{z \lambda}\Ls
(z,\lambda)^\top
= \begin{bmatrix}-L & L \end{bmatrix}. \\
\end{aligned}
\end{equation}
The $\distnewtoncont$, or $\danac$, dynamics are given by
\begin{equation}\label{eq:approx-newt-dynamics}
\begin{aligned}
\dot{z} &= -A_q \nabla_z \Ls (z,\lambda), \\
\dot{\lambda} &= \left[ \nabla_\lambda \Ls (z,\lambda)  \right]^+_\lambda .
\end{aligned}
\end{equation}
The descent in the primal variable $z$ is the approximate Newton
direction as~\eqref{eq:short algorithm}, augmented with dual ascent
dynamics in $\lambda$ (one-hop communication) and implemented in
continuous time. The projection on the dynamics in $\lambda$ ensures
that if $\lambda_i(t_0) \geq 0 $ then $ \lambda_i(t) \geq 0$ for all
$t \geq t_0$.

Define $\mathscr{Z}_q: \real^n \times \real^{2n}_{\geq 0} \rightarrow \real^n
\times \real^{2n}$ as the map in~\eqref{eq:approx-newt-dynamics}
implemented by $\danac$. We now make the following assumptions on
initial conditions and the feasibility set.
\begin{assump}\longthmtitle{Initial Dual Feasibility}\label{ass:init-cond-cont-time}
	The initial condition $\lambda(0)$ is dual feasible,
	i.e. $\lambda(0) \succeq 0$.
\end{assump}
\begin{assump}\longthmtitle{Nontrivial Primal Feasibility}\label{ass:nontrivial-sol}
	The feasibility set of $\Pc 2$ is such that $\exists z$ with $P(z) \prec 0$.
\end{assump}
The dynamics $\mathscr{Z}_q$ are not well suited to handle $\lambda$
infeasible, so Assumption~\ref{ass:init-cond-cont-time} is necessary. As for Assumption~\ref{ass:nontrivial-sol}, if it does
not hold, then either $d = \sum \underline{x}$ or $d = \sum
\overline{x}$ or $\Pc 1$ is infeasible, which are trivial
cases. Assuming it does hold, Slater's condition is satisfied and KKT
conditions are necessary and sufficient for solving $\Pc 2$.

Due to the structure of $L$, $\dot{z}$ is computed using only
$(2q+1)$-hop neighbor information. In practice, the quantity $A_q
\nabla_z\mathscr{L}(z,\lambda)$ may be computed recursively over
multiple one-hop or two-hop rounds of communication, with a discrete
step taken in the direction indicated by
$(\dot{z},\dot{\lambda})$. Note that a table statement of this discretized algorithm would be quite similar to Algorithm~\ref{alg:approx-newton} (with the addition of one-hop dynamics in $\lambda$), so we omit it here for brevity.
Discrete-time algorithms to solve this problem do exist, see e.g.~\cite{ER-SM:16-ocam} in which the authors achieve convergence to a ball around the optimizer whose radius is a function of the step size. However, the analysis of discrete-time algorithms to solve $\Pc 2$ via a Newton-like method is outside the scope of this work.
\subsection{Convergence Analysis}
This section provides a rigorous proof of convergence of the
distributed dynamics $\mathscr{Z}_q$ to the optimizer $(z^\star
,\lambda^\star)$ of $\Pc 2$. The solution $x^\star$ to $\Pc 1$ may
then be computed via a one-hop neighbor communication by $x^\star =
x^0 + Lz^\star$.
\begin{thrm}\longthmtitle{Convergence of Continuous Dynamics
    $\mathscr{Z}_q$} \label{thm:cvg-dana-c}\rm
  If Assumption~\ref{ass:conn-graph}, on the undirected and connected
  graph, Assumption~\ref{ass:initial}, on the feasible initial
  condition, Assumption~\ref{ass:e-vals}, on convergent eigenvalues,
  Assumption~\ref{ass:quad-cost}, on quadratic cost functions,
  Assumption~\ref{ass:init-cond-cont-time}, on the feasible dual
  initial condition, and Assumption~\ref{ass:nontrivial-sol}, on
  nontrivial primal feasibility,
	hold, then the solution trajectories under $\mathscr{Z}_q$
	assymptotically converge to an optimal point $(z^\star ,
	\lambda^\star)$ of $\Pc 2$, where $z^\star$ uniquely satisfies
	$\ones[n]^\top z^\star = \ones[n]^\top z(0)$.
\end{thrm}

\begin{proof}
	Consider $Q = \begin{bmatrix} A_q^{-1} & 0 \\ 0 &
	I_{2n} \end{bmatrix} \succ 0$ and define the Lyapunov function
	\begin{equation}\label{eq:vq}
	\begin{aligned}
	V_Q(z,\lambda) &:= \dfrac{1}{2}\begin{bmatrix}
	z - z^\star \\
	\lambda - \lambda^\star
	\end{bmatrix}^\top Q \begin{bmatrix}
	z - z^\star \\ 
	\lambda - \lambda^\star
	\end{bmatrix} \\
	&= \dfrac{1}{2}\Big(\|
	A_q^{-1/2}(z-z^\star)\|_2^2 + \| (\lambda -
	\lambda^\star)\|_2^2\Big).
	\end{aligned}
	\end{equation}
	The time derivative of $V_Q$ along the trajectories of $\mathscr{Z}_q$ is 
	\begin{equation}\label{eq:vq-dot}
	\begin{aligned}
	\dot{V}_Q(z,\lambda) &= \begin{bmatrix}
	z - z^\star \\
	\lambda - \lambda^\star
	\end{bmatrix}^\top Q \begin{bmatrix}
	\dot{z} \\
	\dot{\lambda}
	\end{bmatrix}  \\
	&= -(z-z^\star)^\top A_q^{-1}A_q \nabla_z \Ls(z,\lambda) \\
	& \qquad+ (\lambda-\lambda^\star)^\top \left[ \nabla_\lambda \Ls(z,\lambda) \right]^+_\lambda \\
	&\overset{(a)}{\leq} -(z-z^\star)^\top \nabla_z \Ls(z,\lambda) + (\lambda-\lambda^\star)^\top \nabla_\lambda \Ls(z,\lambda) \\
	&\overset{(b)}{=} -(z-z^\star)^\top LHL (z-z^\star) \\
	& \qquad - (z-z^\star)^\top \begin{bmatrix} -L & L \end{bmatrix} (\lambda - \lambda^\star) \\
	& \qquad + (\lambda - \lambda^\star)^\top \begin{bmatrix} -L & L \end{bmatrix}^\top (z-z^\star) \\
	&= -\| H^{1/2}L (z-z^\star)\|_2^2 \overset{(c)}{<} 0, \ z\neq z^\star.
	\end{aligned}
	\end{equation}
	The inequality (a) follows from the componentwise relation
	$(\lambda_i - \lambda_i^\star)(\left[ \nabla_{\lambda_i} \Ls
	\right]^+_{\lambda_i} - \nabla_{\lambda_i} \Ls ) \leq 0$. To see
	this, if $\lambda_i > 0$, the projection is inactive and this term
	equals zero. If $\lambda_i = 0$, then the inequality follows from
	$\lambda_i^\star \geq 0$ and $\left[ \nabla_{\lambda_i} \Ls
	\right]^+_{\lambda_i} - \nabla_{\lambda_i} \Ls \geq 0$. The equality
	(b) is obtained from an application of the Fundamental Theorem of
	Calculus and computing the line integral along the line
	$(z(s),\lambda(s)) = s(z,\lambda) + (1-s)(z^\star,\lambda^\star)$ as
	follows:
	\begin{equation*}
	\begin{aligned}
	\nabla_z \Ls(z,\lambda) &= \int_0^1 
	\Big(\nabla_{zz}\Ls(z(s),\lambda(s))(z-z^\star) \\
	&\qquad + \nabla_{\lambda z}\Ls(z(s),\lambda(s))
	(\lambda-\lambda^\star)\Big)ds \\
	&= \nabla_{zz}\Ls(z,\lambda)(z-z^\star) + 
	\nabla_{\lambda z}\Ls(z,\lambda)(\lambda-\lambda^\star), \\
	\nabla_\lambda \Ls(z,\lambda)
	&= \int_0^1 
	\Big(\nabla_{\lambda\lambda}\Ls(z(s),\lambda(s)(\lambda-\lambda^\star) \\
	&\qquad + \nabla_{z\lambda}\Ls(z(s),\lambda(s))(z-z^\star)\Big)ds \\
	&= \nabla_{z\lambda}\Ls(z,\lambda)(z-z^\star),
	\end{aligned}
	\end{equation*}
	where the integrals can be simplified due to $\nabla_{zz}\Ls$ and $\nabla_{\lambda z}\Ls$ constant, as per~\eqref{eq:equiv}. Recalling Remark~\ref{rem:init-traj-sol}, which applies similarly
	here, and noticing $\dot{z}\perp\ones[n]$, it follows from the theorem
	statement that $(z-z^\star)\perp \ones[n]$. Additionally, zero is
	a simple eigenvalue of $H^{1/2}L$ with a corresponding right
	eigenvector $\ones[n]$, implying that (c), the last line of~\eqref{eq:vq-dot}, is strict for $z\neq
	z^\star$.
	
	Let $\mathcal{S} := \setdefB{(z,\lambda)}{z = z^\star ,
		\lambda \succeq 0}$ be
        an asymptotically stable set under the dynamics
        $\mathscr{Z}_q$ defined
        in~\eqref{eq:approx-newt-dynamics}. 
        We aim to show the largest invariant set contained in $\mathcal{S}$ is the optimizer $\{(z^\star,\lambda^\star)\}$,
        so we reason with KKT conditions to
        complete the convergence argument for $\lambda$. For
        $(z,\lambda)\in\mathcal{S}$, clearly primal feasibility is
        satisfied. Assumption~\ref{ass:init-cond-cont-time} gives
        feasibility of $\lambda (0)$, which is maintained along the
        trajectories of $\mathscr{Z}_q$. The stationarity condition
        $\nabla_z \Ls(z^\star,\lambda^\star) = 0$ is also satisfied
        for $(z,\lambda)\in\mathcal{S}$: 
	examine the dynamics $\dot{z}(t) =
	-A_q \nabla_z \Ls(z,\lambda) \equiv 0$.
	It follows that $\nabla_z \Ls(z,\lambda)_{(z,\lambda)\in\mathcal{S}} = 0$ due to $A_q$ being full rank. Then, each KKT condition has been
	satisfied for $(z,\lambda)\in\mathcal{S}$ except
	complementary slackness: $P_i(z) \lambda_i = 0$ for $i \in \until{2n}$. We now address this.
	
	Notice the relation $\dot{z} \equiv 0$
          implies
	\begin{equation}\label{eq:lambda-all-ones}
	\lambda(t) = \hat{\lambda} + \phi_{\underline{\lambda}} (t)
	\begin{bmatrix}
	\ones[n] \\ \zeros[n]
	\end{bmatrix}
	+ \phi_{\overline{\lambda}} (t)
	\begin{bmatrix}
	\zeros[n] \\ \ones[n]
	\end{bmatrix}
	\end{equation}
	for some constant $\hat{\lambda}\in\real^{2n}$ and possibly time
	varying $\phi_{\underline{\lambda}} (t), \phi_{\overline{\lambda}}
	(t)\in\real$. This is due to $\nulo{L} = \spn{\{\ones[n]\}}$ and
	inferring from $\dot{z} \equiv 0$ that $\begin{bmatrix} -L & L
	\end{bmatrix}\lambda (t)$ must be constant. 
	Additionally, we may infer from the map $\mathscr{Z}_q$ that
	$\phi_{\underline{\lambda}} (t), \phi_{\overline{\lambda}}
	(t)$ are continuous and piecewise smooth. The
	dynamics $\dot{\lambda}$ and
	differentiating~\eqref{eq:lambda-all-ones} in time gives
	\begin{equation}\label{eq:lambda-dot-all-ones}
	\begin{aligned}
	\dot{\lambda} = \left[ \nabla_\lambda \Ls(z^\star,\lambda)
	\right]^+_\lambda &= \left[ P(z^\star)  \right]^+_\lambda \\
	&\in \partial \phi_{\underline{\lambda}} (t)
	\begin{bmatrix}
	\mathbf{1}_{n} \\ \zeros[n]
	\end{bmatrix} 
	+ \partial \phi_{\overline{\lambda}} (t)
	\begin{bmatrix}
	\mathbf{0}_{n} \\ \mathbf{1}_n
	\end{bmatrix},
	\end{aligned}
	\end{equation}
	where $\partial \phi_{\underline{\lambda}} (t)$ and $\partial
	\phi_{\overline{\lambda}} (t)$ are subdifferentials with respect to time
	of
	$\phi_{\underline{\lambda}} (t)$ and $\phi_{\overline{\lambda}} (t)$,
	respectively. Then, $\phi_{\underline{\lambda}} (t)$ and
	$\phi_{\overline{\lambda}} (t)$ are additionally piecewise linear due
	to $P(z^\star)$ constant.
	We now state two cases for $\underline{P}(z^\star)$ to prove
	$\underline{\lambda}(t)\rightarrow\underline{\lambda}^\star$.
	
	\textbf{Case 1:} $\underline{P}_i(z^\star) = 0$ for at least one
	$i\in\until{n}$. Then, $\dot{\underline{\lambda}}_i = 0$ and
	from~\eqref{eq:lambda-dot-all-ones} this implies
	$\dot{\underline{\lambda}} = \mathbf{0}_{n}$. Reasoning from the
	projection dynamics, this implies either $\underline{\lambda}_j = 0$
	or $\underline{P}_j(z^\star) = 0$ for each $j$, which satisfies the
	complementary slackness condition $\underline{\lambda}_j^\star
	\underline{P}_j(z^\star) = 0$ for every $j\in\until{n}$, and we
	conclude that $\underline{\lambda} = \underline{\lambda}^\star$ for
	$(z,\lambda)\in\mathcal{S}$.
	
	\textbf{Case 2:} $\underline{P}(z^\star)\prec
	0$.
	Complementary slackness states $\underline{\lambda}_i^\star
	\underline{P}_i(z^\star) = 0$ for each $i\in\until{n}$, implying
	$\underline{\lambda}^\star = \zeros[n]$. The dynamics preserve
	$\lambda(t) \succeq 0$, so the quantity $\underline{\lambda}_i -
	\underline{\lambda}_i^\star$ is strictly positive for any
	$\underline{\lambda}_i \neq \underline{\lambda}_i^\star$. Applying this to the term $(\lambda-\lambda^\star)^\top[\nabla_\lambda\Ls(z,\lambda)]_\lambda^+$ obtained from the second equality (third line) of~\eqref{eq:vq-dot}, and also applying
	$\underline{P}(z^\star)=\nabla_\lambda\Ls(z^\star,\lambda)\prec 0$, we obtain $\dot{V}_Q < 0$ for $z =
	z^\star, \underline{\lambda} \neq
	\underline{\lambda}^\star$.
	
	The inferences of Case 1 (satisfying complementary slackness)
        and Case 2 (reasoning with $\dot{V}_Q$) hold similarly for
        $\overline{\lambda}$. Then, we have shown that
        $\dot{V}_Q(z,\lambda) < 0, \forall (z,\lambda)\in\mathcal{S}\setminus\{(z^\star,\lambda^\star)\}$.
	Asymptotic convergence to the primal and dual optimizers of
        $\Pc 2$ follows from the LaSalle Invariance Principle~\cite{HKK:02}.
\end{proof}
\subsection{Interpretation of the Convergence Result}\label{ssec:interpretation}
For fast convergence, it is desirable for the ratio $\dot{V}_Q / V_Q <
0$ to be large in magnitude for any $(z,\lambda)\in\real^n \times
\real_{+}^{2n}$. Recall the diagonalization of $A_q$ and use this to
compute $A_q^{-1}$:
\begin{align*}
A_q &= W \begin{bmatrix}
\dfrac{1 - \eta_1^{q+1}}{1 - \eta_1} & & & \\
& \ddots & & \\
& & \dfrac{1 - \eta_{n-1}^{q+1}}{1 - \eta_{n-1}} \\
& & & q+1
\end{bmatrix} W^\top ,\displaybreak[0]\\
A_q^{-1} &= W \begin{bmatrix}
\dfrac{1 - \eta_1}{1 - \eta_1^{q+1}} & & & \\
& \ddots & & \\
& & \dfrac{1 - \eta_{n-1}}{1 - \eta_{n-1}^{q+1}} \\
& & & (q+1)^{-1}
\end{bmatrix} W^\top .
\end{align*}
Next, write $z - z^\star = \zeta_1w_1 + \dots + \zeta_{n-1}w_{n-1}$ as
a weighted sum of the eigenvectors $w_i$ of $I_n - LHL$. Note that we
do not need $w_n = \mathbf{1}_{n}$ for this representation due to $z -
z^\star \perp w_n$. Then, $V_Q = \sum_{i=1}^{n-1} \zeta_i^2
(1-\eta_i)/(1-\eta_i^{q+1}) + V_\lambda$, where $V_\lambda :=
\vert\vert \lambda - \lambda^\star \vert\vert_2^2$. Additionally, note
that $LHL$ and $A_q^{-1}$ share eigenvectors, so $\dot{V}_Q \leq
-\sum_{i=1}^{n-1} \zeta_i^2 (1-\eta_i)$. 
Toward this end, we can write
\begin{equation*}
\dfrac{\dot{V}_Q}{V_Q} \leq \dfrac{-\sum_{i=1}^{n-1} 
	\zeta_i^2(1- \eta_i)}{\sum_{i=1}^{n-1} \zeta_i^2\left(\dfrac{1- 
		\eta_i}{1-\eta_i^{q+1}}\right) + V_\lambda}.
\end{equation*}
To interpret this, first reason
with the values of $q$. Consider $q=0$, which is analogous to a
gradient-based method. Then, the rational in the sum contained in the
denominator is equal to one and there is no \emph{weighting}, in a
sense, to the step direction. In other words, if the value of
$\zeta_i$ happens to be large in magnitude corresponding to the
eigenvector $w_i$ of $\nabla_{zz}\mathscr{L}$ whose corresponding
eigenvalue $(1-\eta_i)$ is small in magnitude, then that term does not
appropriately dominate the numerator relative to each other term and
the quantity $\dot{V}_Q / V_Q$ is small in magnitude. On the other hand,
if
$q$ is large, then the quantity $1-\eta_i^{q+1}$ is close to $1$, and
the terms of the sums in the numerator and denominator have the effect
of ``cancelling'' one another, which provides more uniform convergence
on the trajectories of $z$. 
In addition, if the values of $\eta_i$ are small in magnitude,
i.e. our weight design on $L$ was relatively successful, the quantity $1-\eta_i^{q+1}$ approaches $1$ more quickly
and the effect of a particular $\zeta_i$ being large relative to the
other terms in the sum is diminished for any particular $q$.

Note that, although we have framed this argument as an improvement
over the gradient technique, it may be the case that for a particular
time $t$ the decomposition on $z(t)$ may have a large $\zeta_i$
corresponding to $1-\eta_i$ large. This actually provides superior
momentary convergence compared to a Newton-like method. However, we
contend that the oscillatory nature of the trajectories over the
entire time horizon gives way to improved convergence from the Newton
flavor of our algorithm. This is confirmed in simulation.

Finally, it is apparent that choosing $q$ even is (generally speaking)
superior to $q$ odd: the quantity $1-\eta_i^{q+1}$ may take values in
$\left[ 1-\varepsilon^{q+1}, 1+\varepsilon^{q+1}\right]$, as opposed
to odd $q$ for which $1-\eta_i^{q+1}$ takes values in $\left[
1-\varepsilon^{q+1}, 1\right]$. We would like this quantity to be
large so the magnitude of $\dot{V}_Q /V_Q$ is large. This observation
of choosing even $q$ to prompt superior convergence is confirmed in
simulation.

This discussion neglects the $V_\lambda$ term which may be large
for arbitrarily "bad" initial conditions $\lambda(0) \succeq
0$. However, the ascent direction in $\lambda$
is clearly more effective for $z$ nearly optimal, so this term is
``cooperative'' in the sense that its decay roughly corresponds to
the decay of the Lyapunov term in $z$.

To summarize, gradient methods neglect the curvature of the underlying
cost function, which dictates the convergence properties of descent
algorithms. By weighting the descent direction by $A_q$, we elegantly
capture this curvature in a distributed fashion and the solution
trajectory reflects this property. We now provide a remark on
convergence of the algorithm for nonquadratic costs that are well approximated by quadratic functions.

\begin{rem}\longthmtitle{Convergence of $\danac$ for Approximately Quadratic Costs}
	Instead of Assumption~\ref{ass:quad-cost} (quadratic costs), let
	Assumption~\ref{ass:cost} (general costs) hold and consider the
	dynamics
	\begin{equation}\label{eq:cont-dyn-nonquadr}
	\begin{aligned}
	\dot{z} &= -A_q(z) \nabla_z \Ls (z,\lambda), \\
	\dot{\lambda} &= \left[ \nabla_\lambda \Ls (z,\lambda)  \right]^+_\lambda .
	\end{aligned}
	\end{equation}
	Let $H^\prime:= \dfrac{H_\Delta + H_\delta}{2}$ and
	$A_q^\prime:=\sum_{p=0}^q (I_n - LH^\prime L)^p$. In a sense, these
	matrices are obtained from quadratic approximations of the
	nonquadratic costs $f_i$, i.e. $\left| \dfrac{\partial^2
		f_i}{\partial x_i^2} - H^\prime_{ii} \right| \leq
	\dfrac{\Delta_i - \delta_i}{2}$. Use $Q = \begin{bmatrix}
	A_q^{\prime -1} & 0 \\ 0 & I_{2n} \end{bmatrix}$ to define the
	quadratic Lyapunov function $V_Q(z,\lambda)$ as
	in~\eqref{eq:vq}. Differentiating along the trajectories
	of~\eqref{eq:cont-dyn-nonquadr} now gives
	\begin{equation*}
	\dot{V}_Q(z,\lambda) = \dot{V}^\prime_Q(z,\lambda) + U(e,z,\lambda),
	\end{equation*}
	where $e$ gives some measure of how much the functions deviate from
	quadratic and $U(0,z,\lambda) = 0$. The
	$\dot{V}^\prime_Q(z,\lambda)$ is obtained by decomposing the
	dynamics~\eqref{eq:cont-dyn-nonquadr} as
	\begin{equation*}
	\begin{aligned}
	\dot{z} &= -A_q^\prime \nabla_z \Ls (z,\lambda) + u(e,z,\lambda), \\
	\dot{\lambda} &= \left[ \nabla_\lambda \Ls (z,\lambda)
	\right]^+_\lambda .
	\end{aligned}
	\end{equation*}
	and including only the terms without $u(e,z,\lambda)$, where the
	remaining terms are captured by $U(e,z,\lambda)$. $U$ and $u$
	are continuous functions of $e$, and $u(0,z,\lambda) = 0$. Applying
	the convergence argument of Theorem~\ref{thm:cvg-dana-c} to
	$V^\prime_Q(z,\lambda)$, the continuity of $U$ and $u$ imply
	$\dot{V}^\prime_Q(z,\lambda) < -U(\bar{e},z,\lambda)$ for
	sufficiently small $\bar{e}$. Therefore, $\dot{V}_Q(z,\lambda) < 0$
	for functions that are well approximated by quadratic functions.
\end{rem}
\section{Simulations and Discussion}		\label{sec:sims-discuss}
In this section, we implement our weight design and verify the
convergence of the $\distnewton$ algorithm in each of the discrete-time (relaxed) and continuous-time (box-constrained) settings.
\subsection{Weight Design}
To evaluate the weight design posed in
Section~\ref{sec:laplacian-design} we use quadratic costs in
accordance with Assumption~\ref{ass:quad-cost}, i.e. $\delta_i = \Delta_i = a_i, \forall i$. We do this in order to isolate the
other parameters for this part of the study. Consider the following metrics: the
solution to $\Pc 4$ followed by the post-scaling by $\beta$ gives
$\varepsilon_{L^\star} := \max(\vert 1 - \mu_i(M^\star)\vert )$; this
metric represents the convergence speed of $\distnewton$ when applying
our proposed weight design of $L$.  Using the same topology $(\nodes,
\mathcal{E})$, the solution to $\Pc 5$ gives the metric
$\varepsilon_A$. Note that $\varepsilon_A$ is a \emph{best-case}
estimate of the weight design problem; however, ``reverse engineering" an $L^\star$ from the solution $A^\star$ to $\Pc 5$ is both intractable and generally likely to be infeasible. With this in mind, the metric $\varepsilon_A$ is a very conservative lower bound, whereas $\varepsilon_{L^\star}$ is the metric for which we can
compute a feasible $L^\star$. The objective of each problem is to minimize the
associated $\varepsilon$; to this end, we aim to characterize the
relationship between network parameters and these metrics. We ran 100
trials on each of 16 test cases which encapsulate a variety of
parameter cases: two cases for the cost coefficients, a \textit{tight}
distribution $a_i \in \mathcal{U}\left[ 0.8, 1.2\right]$ and a
\textit{wide} distribution $a_i \in \mathcal{U}\left[ 0.2,
5\right]$. For topologies, we randomly generated connected graphs
with network size $n \in \{10, 20, 30, 40, 50\}$, a \textit{linearly}
scaled number of edges $\vert\mathcal{E}\vert = 3n$, and a
\textit{quadratically} scaled number of edges $\vert\mathcal{E}\vert =
0.16n^2$ for $n \in \{30, 40, 50\}$. The linearly scaled
connectivity case corresponds to keeping the average degree of a node
constant for increasing network sizes, while the quadratically scaled
case roughly preserves the proportion of connected edges to total
possible edges, which is a quadratic function of $n$ and equal to
$n(n-1)/2$ for an undirected network. The results are depicted in
Table~\ref{L-design-table}, where the quadratically scaled cases are indicated by boldface. This gives the mean $\Sigma$ and standard
deviation $\sigma$ of the distributions for \emph{performance}
$\varepsilon_{L^\star}$ and \emph{performance gap}
$\varepsilon_{L^\star} - \varepsilon_A$.
\begin{table}[]
	\centering
	\caption{Laplacian Design. Quadratically-scaled number-of-edge cases are indicated by boldface.}
	\label{L-design-table}
	\begin{tabular}{|c||c|c|c|c|}
		\hline
		\begin{tabular}[c]{@{}c@{}} $a_i \in \mathcal{U} \left[ 0.8, 1.2\right]$ \\ $b_i \in \mathcal{U} \left[ 0, 1\right]$\end{tabular} & $\Sigma ( \varepsilon_{L^\star})$ & $\sigma (\varepsilon_{L^\star})$ & $\Sigma ( \varepsilon_{L^\star} - \varepsilon_A)$ & $\sigma ( \varepsilon_{L^\star} - \varepsilon_A)$ \\ \hline
		\begin{tabular}[c]{@{}c@{}}$n = 10$ \\ $\vert \mathcal{E} \vert = 30$ \end{tabular}  & 0.6343 & 0.0599 & 0.2767 & 0.0186 \\ \hline
		\begin{tabular}[c]{@{}c@{}}$n = 20$ \\ $\vert \mathcal{E} \vert = 60$ \end{tabular}  & 0.8655 & 0.0383 & 0.2879 & 0.0217 \\ \hline
		\begin{tabular}[c]{@{}c@{}}$n = 30$ \\ $\vert \mathcal{E} \vert = 90$ \end{tabular}  & 0.9100 & 0.0250 & 0.2666 & 0.0233 \\ \hline
		\begin{tabular}[c]{@{}c@{}}$n = 40$ \\ $\vert \mathcal{E} \vert = 120$ \end{tabular}  & 0.9303 & 0.0201 & 0.2501 & 0.0264 \\ \hline
		\begin{tabular}[c]{@{}c@{}}$n = 50$ \\ $\vert \mathcal{E} \vert = 150$ \end{tabular}  & 0.9422 & 0.0175 & 0.2375 & 0.0264 \\ \hline
		\begin{tabular}[c]{@{}c@{}}$n = 30$ \\ $\mathbf{\vert \mathcal{E} \vert = 144}$ \end{tabular}  & 0.7266 & 0.0324 & 0.2973 & 0.0070 \\ \hline
		\begin{tabular}[c]{@{}c@{}}$n = 40$ \\ $\mathbf{\vert \mathcal{E} \vert = 256}$ \end{tabular}  & 0.6528 & 0.0366 & 0.2829 & 0.0091 \\ \hline
		\begin{tabular}[c]{@{}c@{}}$n = 50$ \\ $\mathbf{\vert \mathcal{E} \vert = 400}$ \end{tabular}  & 0.5840 & 0.0281 & 0.2641 & 0.0101 \\ \hline \hline
		\begin{tabular}[c]{@{}c@{}} $a_i \in \mathcal{U} \left[ 0.2, 5\right]$ \\ $b_i \in \mathcal{U} \left[ 0, 1\right]$\end{tabular} & $\Sigma ( \varepsilon_{L^\star})$ & $\sigma (\varepsilon_{L^\star})$ & $\Sigma ( \varepsilon_{L^\star} - \varepsilon_A)$ & $\sigma ( \varepsilon_{L^\star} - \varepsilon_A)$ \\ \hline
		\begin{tabular}[c]{@{}c@{}}$n = 10$ \\ $\vert \mathcal{E} \vert = 30$ \end{tabular}  & 0.6885 & 0.0831 & 0.3288 & 0.0769 \\ \hline
		\begin{tabular}[c]{@{}c@{}}$n = 20$ \\ $\vert \mathcal{E} \vert = 60$ \end{tabular}  & 0.8965 & 0.0410 & 0.3241 & 0.0437 \\ \hline
		\begin{tabular}[c]{@{}c@{}}$n = 30$ \\ $\vert \mathcal{E} \vert = 90$ \end{tabular}  & 0.9389 & 0.0254 & 0.2878 & 0.0395 \\ \hline
		\begin{tabular}[c]{@{}c@{}}$n = 40$ \\ $\vert \mathcal{E} \vert = 120$ \end{tabular} & 0.9539 & 0.0189 & 0.2830 & 0.0355 \\ \hline
		\begin{tabular}[c]{@{}c@{}}$n = 50$ \\ $\vert \mathcal{E} \vert = 150$ \end{tabular} & 0.9628 & 0.0168 & 0.2590 & 0.0335 \\ \hline
		\begin{tabular}[c]{@{}c@{}}$n = 30$ \\ $\mathbf{\vert \mathcal{E} \vert = 144}$ \end{tabular} & 0.7997 & 0.0520 & 0.3587 & 0.0524 \\ \hline
		\begin{tabular}[c]{@{}c@{}}$n = 40$ \\ $\mathbf{\vert \mathcal{E} \vert = 256}$ \end{tabular} & 0.7339 & 0.0550 & 0.3688 & 0.0569 \\ \hline
		\begin{tabular}[c]{@{}c@{}}$n = 50$ \\ $\mathbf{\vert \mathcal{E} \vert = 400}$ \end{tabular} & 0.6741 & 0.0487 & 0.3543 & 0.0425 \\ \hline
	\end{tabular}
\end{table}

From these results, first note that the \emph{tightly} distributed
coefficients $a_i$ result in improved $\varepsilon_{L^\star}$ across
the board compared to the \emph{widely} distributed coefficients. We
attribute this to the approximation $LHL \approx
\left(\dfrac{\sqrt{H}L + L\sqrt{H}}{2}\right)^2$ being more accurate
for roughly homogeneous $H = \diag{a_i}$. Next, it is clear that in
the cases with \emph{linearly} scaled edges, $\varepsilon_{L^\star}$
worsens as network size increases.  This is intuitive: the
\textit{proportion} of connected edges in the graph decreases as
network size increases in these cases. This also manifests itself in
the performance gap $\varepsilon_{L^\star} - \varepsilon_A$ shrinking,
indicating the \emph{best-case} solution $\varepsilon_A$ (for which a
valid $L$ does not necessarily exist) degrades even quicker as a
function of network size than our solution $\varepsilon_{L^\star}$. On
the other hand, $\varepsilon_{L^\star}$ substantially improves as
network size increases in the \emph{quadratically} scaled cases, with
a roughly constant performance gap $\varepsilon_{L^\star} -
\varepsilon_A$. Considering this relationship between the linear and
quadratic scalings on $\vert \mathcal{E}\vert$ and the metrics
$\varepsilon_{L^\star}$ and $\varepsilon_A$, we get the impression
that both proportion of connectedness and average node degree play a
role in both the effectiveness of our weight-designed solution
$L^\star$ and the best-case solution.  For this reason, we postulate
that $\varepsilon_{L^\star}$ remains roughly constant in large-scale
applications if the number of edges is scaled subquadratically as a
function of network size; equivalently, the convergence properties of
$\distnewton$ algorithm remain relatively unchanged when using our
proposed weight design and growing the number of communications per
agent sublinearly as a function of $n$.

\subsection{Discrete-Time Distributed Approx-Newton}

Consider solving $\Pc 6$ with $\danad$ for a network of $n = 100$ generators and $\vert \mathcal{E}
\vert = 250$ communication links. The local computations required of each generator
are simple vector operations whose dimension scales linearly with the
network size, which can be implemented on a microprocessor. The graph
topology is plotted in Figure~\ref{fig:graph}. The problem parameters are given by
\begin{equation*}
\begin{aligned}
&f_i(x_i) = \dfrac{1}{2}a_i x_i^2 + b_i x_i + c_i \sin{(x_i + \theta_i)}, \\
&a_i\in \mathcal{U} [2,4], \quad b_i\in \mathcal{U} [-1,1], \\ 
&c_i\in \mathcal{U} [0,1], \quad \theta_i\in \mathcal{U} [0,2\pi], \\
&d = 200, \quad x^0 = (d/n)\ones[n].
\end{aligned}
\end{equation*}
Note that $0 < a_i - c_i \leq \dfrac{\partial^2 f_i}{\partial x_i^2}
\leq a_i + c_i$ satisfies Assumption~\ref{ass:cost}. We compare to
the $\dgd$ and weight design policies for resource allocation described in~\cite{LX-SB:06}, along with 
an ``unweighted'' version of~\cite{LX-SB:06} in the sense that
$L$ is taken to be the degree matrix minus the adjacency matrix of the
graph, followed by the post-scaling described in
Section~\ref{sec:topology-design} to guarantee convergence. The
results are given in Figure~\ref{fig:qcomp}, which show linear
convergence to the optimal value as the number of iterations
increases, with fewer iterations needed for larger $q$. We note a
substantially improved convergence over the $\dgd$ methods, even for the
$q=0$ case which utilizes an equal number of agent-to-agent
communications as $\dgd$. This can be attributed in-part to the superior weight design of our method, which is cognizant of second-order information.

In addition, in Figure~\ref{fig:qcomp} we plot convergence of DGD, weighted by the one-sided design scheme in~\cite{LX-SB:06}, compared to our two-sided design with $q=0$, for cases in which only a universal bound on $\delta_i$, $\Delta_i$ is known (namely, using $\underline{\delta} \leq \delta_i, \Delta_i \leq \overline{\Delta}, \forall i$, as in Remark~\ref{rem:glob-hess-bound}). We note an improved convergence in each case for the locally known bounds versus the universal bound, while the locally weighted DGD method outperforms our $q=0$ two-sided globally weighted method by a slight margin.

\begin{figure}[h]
	\centering
	\includegraphics[scale = 0.35]{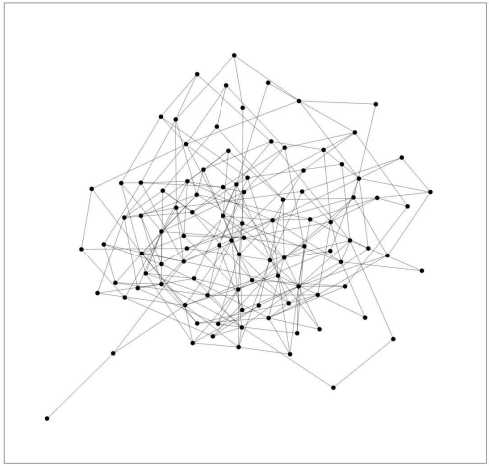}
	\caption{Communication topology used for discrete-time numerical study; $n=100,\vert\E\vert = 250$.}
	\label{fig:graph}
\end{figure}
\begin{figure}[h]
	\centering
	\includegraphics[scale = 0.45]{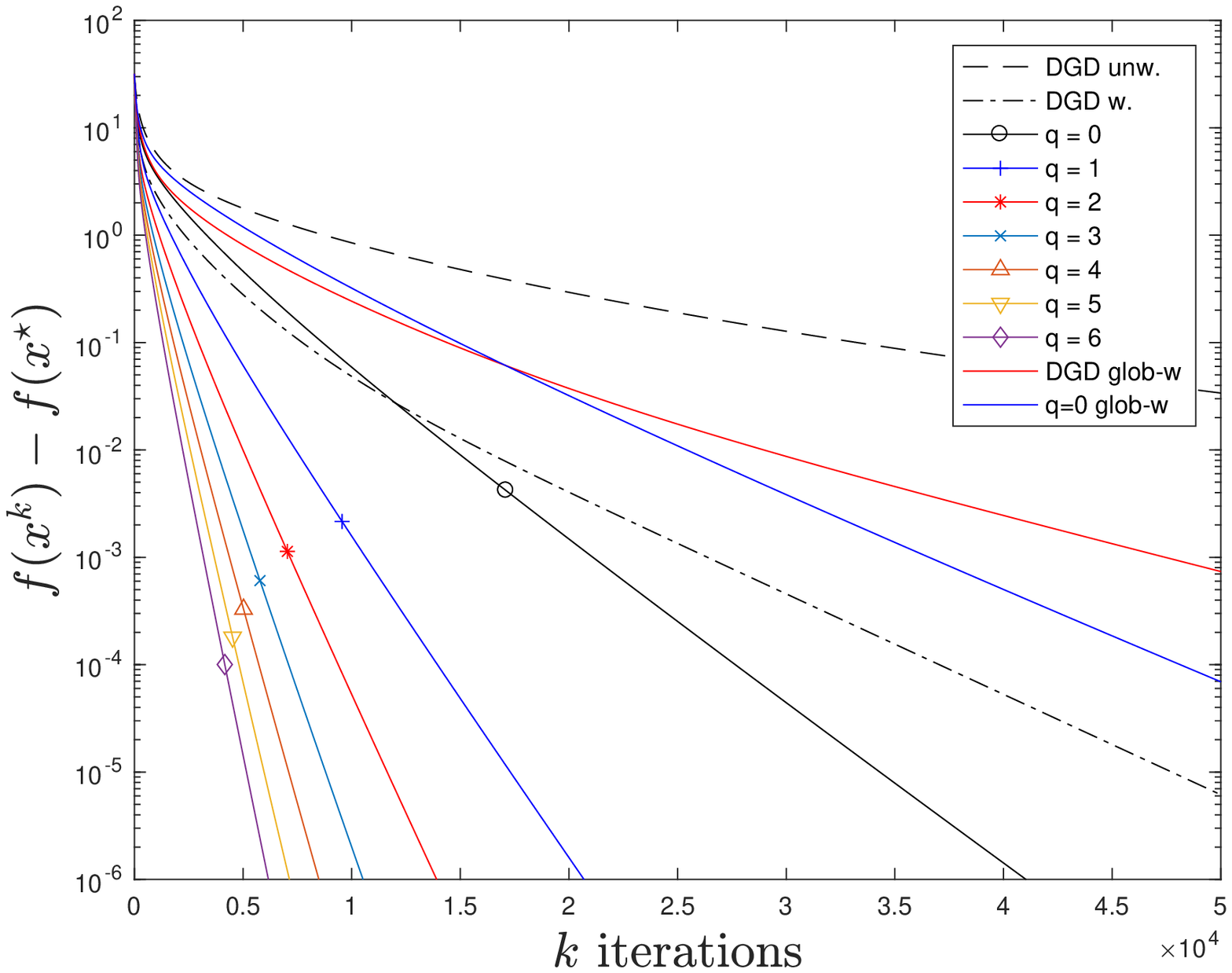}
	\caption{Comparison of weighted and unweighted $\dgd$ versus
		$\distnewtondisc$ with various $q$ for solving $\Pc 6$; $n=100,\vert\E\vert = 250$.}
	\label{fig:qcomp}
\end{figure}
\subsection{Continuous-Time Distributed Approx-Newton}
We now study $\danac$ for solving $\Pc 1$ for
a simple $3$ node network with two edges $\E = \{\{1,2\},
\{2,3\}\}$ for the sake of visualizing trajectories. The problem parameters are given by
\begin{equation*}
\begin{aligned}
&f_1(x_1) = \dfrac{1}{4}x_1^2 + \dfrac{1}{2}x_1, \\
&f_2(x_2) = \dfrac{3}{4}x_2^2 + \dfrac{1}{2}x_2, \\
&f_3(x_3) = 2x_3^2 + \dfrac{1}{2}x_3, \\
&\underline{x} = \begin{bmatrix}0.2 & 2.5 & 1.5\end{bmatrix}^\top,
\quad
\overline{x} = \begin{bmatrix}1 & 6 & 4\end{bmatrix}^\top, \quad d = 6, \\
&x^0 = \begin{bmatrix}5 & -1 & 2\end{bmatrix}^\top, \quad z(0) = \zeros[3], \\
&\underline{\lambda}(0) = \begin{bmatrix} 1.5 & .5 & 0\end{bmatrix},
\quad \overline{\lambda}(0) = \begin{bmatrix}0 & 2 & 1\end{bmatrix}
\end{aligned}
\end{equation*}
Note that $x^0$ is infeasible with respect to
$\underline{x},\overline{x}$; all that we require is it satisfies
Assumption~\ref{ass:initial} (feasible with respect to $d$). We plot the
trajectories of the $3$-dimensional state projected onto the plane
orthogonal to $\ones[3]$ under various $q$. Figure~\ref{fig:newton-traj} shows this, with a zoomed look at the optimizer in Figure~\ref{fig:newton-traj-zoomed}.

\begin{figure}[h]
	\centering
	\includegraphics[scale = 0.45]{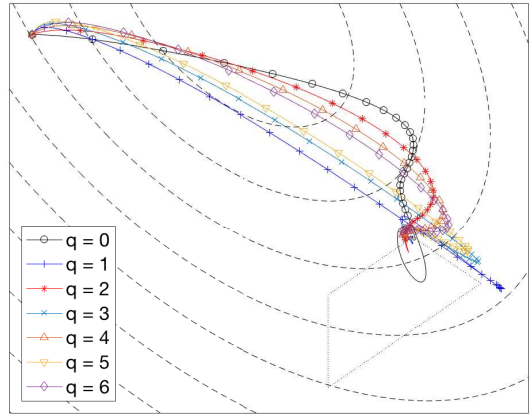}
	\caption{Three node case: projection of $x^0 + Lz(t)\in\real^3$ onto the
		$2$-dimensional plane $\setdef{x}{\sum_i x_i = d}$. Markers
		plotted for $t = 0,0.2,0.4,\dots,5$ seconds. Dashed line
		ellipses indicate intersection of ellipsoid level sets with
		the plane; dotted lines indicate intersection of box
		constraints with the plane.}
	\label{fig:newton-traj}
\end{figure}

\begin{figure}[h]
	\centering
	\includegraphics[scale = 0.45]{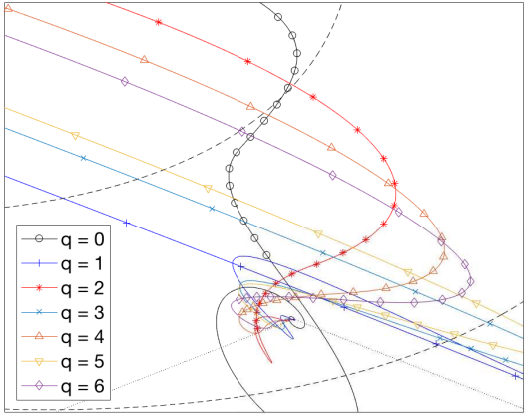}
	\caption{Three node case: trajectories zoomed closer to the optimizer. Markers
		plotted in $0.2s$ increments up to $t = 5s$.}
	\label{fig:newton-traj-zoomed}
\end{figure}

It is clear that choosing $q$ even versus $q$ odd has a qualitative
effect on the shape of the trajectories, as noted in
Section~\ref{ssec:interpretation}. Looking at
Figure~\ref{fig:newton-traj}, it seems the trajectories are intially
pulled toward the unconstrained optimizer (center of the level sets)
with some bias due to $\lambda(0)\neq \zeros[6]$. As $\lambda$ is
given time to evolve, these trajectories are pulled back toward
satisfying the box constraints indicated by the dotted quadrilateral,
i.e. the intersection of the box constraints and the plane defined by
$\setdef{x}{\sum_i x_i = d}$.

For a quantitative comparison, we consider $n = 40$ generators with $\vert\E\vert=156$ communication links whose graph is given by Figure~\ref{fig:graph40} and
the following parameters.
\begin{equation*}
\begin{aligned}
&f_i(x_i) = \dfrac{1}{2}a_i x_i^2 + b_i x_i, \ a_i\in\mathcal{U}[0.5,3], \ b_i\in\mathcal{U}[-2,2],\\
&\underline{x}_i \in\mathcal{U}[1.5,3], \quad
\overline{x}_i \in\mathcal{U}[3,4.5], \quad i\in\until{100}, \\
&d = 120, \ x^0 = 3*\ones[40], \ z(0) = \zeros[40], \ \lambda(0) = \zeros[80].
\end{aligned}
\end{equation*}
\begin{figure}[h]
	\centering
	\includegraphics[scale = 0.45]{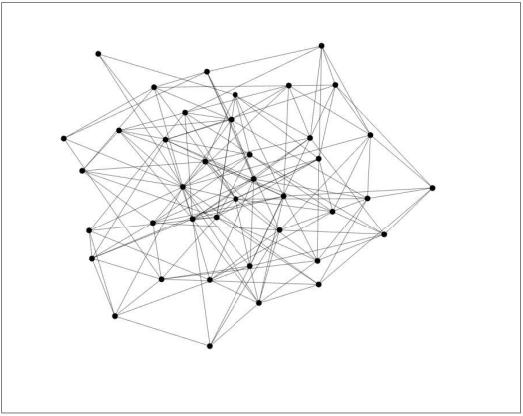}
	\caption{Communication graph for continuous-time numerical study: 40 nodes and 156 edges.}
	\label{fig:graph40}
\end{figure}
\begin{figure}[h]
	\centering
	\includegraphics[scale = 0.45]{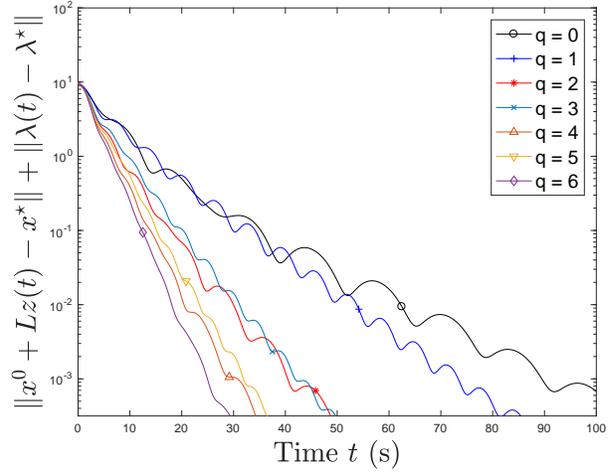}
	\caption{Error in the primal and dual state variables versus time for various $q$; $n=40,\vert\E\vert=156$.}
	\label{fig:errstate_cont}
\end{figure}
\begin{figure}[h]
	\centering
	\includegraphics[scale = 0.45]{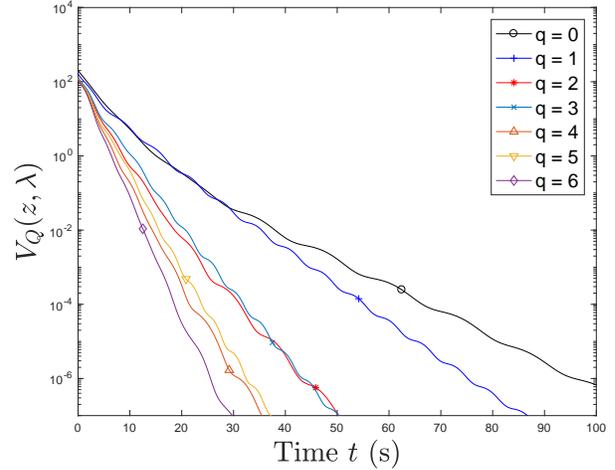}
	\caption{Value of the Lyapunov function $V_Q$ versus time for various $q$; $n=40,\vert\E\vert=156$.}
	\label{fig:Lyap_cont}
\end{figure}
\begin{figure}[h]
	\centering
	\includegraphics[scale = 0.45]{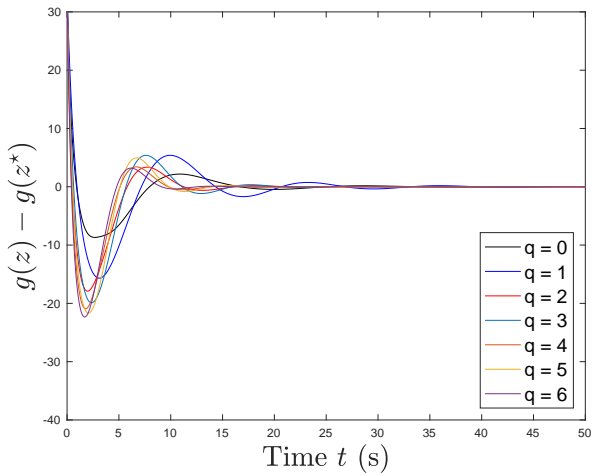}
	\caption{Value of the objective function versus time for various $q$; $n=40,\vert\E\vert=156$.}
	\label{fig:errfn_cont}
\end{figure}

Note from Figure~\ref{fig:errstate_cont} that convergence with respect
to $\Vert x^0 + Lz(t) - x^\star\Vert + \Vert\lambda(t) -
\lambda^\star\Vert$ is not monotonic for some $q$. This is resolved in
Figure~\ref{fig:Lyap_cont} by examining $V_Q$ as defined
by~\eqref{eq:vq}. We also note the phenomenon of faster convergence
for even $q$ over odd $q+1$; the reason for this is related to the
modes of $I_n - LHL$ and was discussed in
Section~\ref{ssec:interpretation}. However, increasing $q$ on a whole
lends itself to superior convergence compared to smaller $q$. As for
the metric $g(z) - g(z^\star)$ in Figure~\ref{fig:errfn_cont}, note
that these values become significantly negative before eventually
stabilizing around zero. The reason for this is simple: in order for
the $\mathscr{Z}_q$ dynamics~\eqref{eq:approx-newt-dynamics} in
$\lambda$ to ``activate," the primal variable must become infeasible
with respect to the box constraints. In this sense, the stabilization
to zero of the plots in Figure~\ref{fig:errfn_cont} represents the
trajectories converging to feasible points of $\Pc 2$.

\subsection{Robust DANA Implementation}\label{ssec:sims-robust-dana}
Lastly, we provide a simulation justification for relaxing
Assumption~\ref{ass:initial} via the method described in
Remark~\ref{rem:robust}. Figure~\ref{fig:robust-state-err} plots the
error in the primal and dual states over time of the modified
``robust" method, which tends to approach zero for all observed values
of $q$, and Figure~\ref{fig:robust-sum-err} demonstrates that the
violation of the equality constraint stablizes to zero very
quickly. Noisy state perturbations are injected at $t=25,50,75$, and
we observe a rapid re-approach to the plane satisfying the equality
constraint. However, even though the algorithm presents a faster
convergence than gradient methods, here do not observe as clear of a
relationship between performance and increased $q$ as in previous
settings. The investigation of the properties of this algorithm is
left as future work.

\begin{figure}[h]
	\centering
	\includegraphics[scale = 0.45]{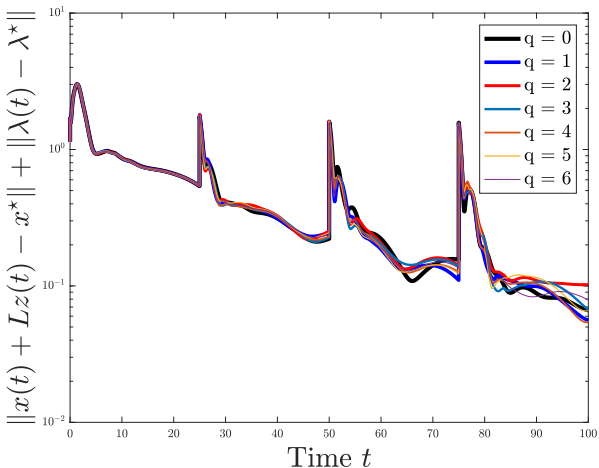}
	\caption{ Error in the primal and dual states for a robust implementation of DANA; $n=20,\vert\E\vert=40$. Initialization does not satisfy Assumption~\ref{ass:initial}, and perturbations are injected at $t=25,50,75$.}
	\label{fig:robust-state-err}
\end{figure}
\begin{figure}[h]
	\centering
	\includegraphics[scale = 0.45]{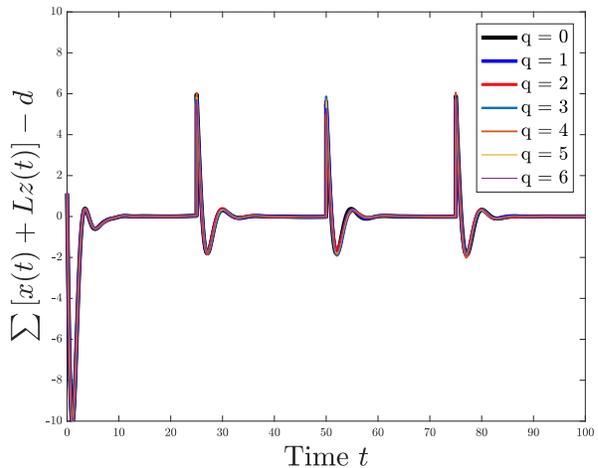}
	\caption{ Violation of the resource constraint over time for robust DANA; $n=20,\vert\E\vert=40$. Perturbations are injected at $t=25,50,75$.}
	\label{fig:robust-sum-err}
\end{figure}

\section{Conclusion and Future Work}
\label{sec:conclusion}

Motivated by economic dispatch problems and separable resource allocation problems in general, this work proposed a class of novel
$\distnewton$ algorithms. We first posed the topology design proplem and provided an
effective method for designing communication weightings.  The
weight design we propose is more cognizant of the problem geometry,
and it outperforms the current literature on network weight
design even when applied to a gradient-like method. Our contribution on the second-order weight design approach is novel but is limited in scope to the given problem formulation. Distributed second-order methods are quite immature in the present literature, so an emphasis of future work is to generalize this weight design notion to a broader class of problems. Ongoing work also includes generalizing the cost functions for box-constrained settings and discretizing the continuous-time algorithm. In addition, we aim to develop distributed Newton-like methods suited to handle more general constraints and design for robustness under uncertain parameters or lossy
communications. Another point of interest is to further study
methods for solving bilinear problems and apply these to weight design within the Newton framework.

\bibliographystyle{abbrv}
\bibliography{alias,SMD-add,SM,JC}

\begin{thebibliography}{10}

\bibitem{AW-BW-GS:12}
G.~S. A.~Wood, B.~Wollenberg.
\newblock {\em Power Generation, Operation, and Control}.
\newblock John Wiley, 3 edition, 2012.

\bibitem{TA-CYC-SM:17a}
T.~Anderson, C.-Y. Chang, and S.~Mart{\'\i}nez.
\newblock Weight design of distributed approximate {N}ewton algorithms for
  constrained optimization.
\newblock In {\em {IEEE} Conference on Control Technology and Applications},
  pages 632--637, Kohala Coast, Hawaii, USA, 2017.

\bibitem{SB-LV:04}
S.~Boyd and L.~Vandenberghe.
\newblock {\em Convex Optimization}.
\newblock Cambridge University Press, 2004.

\bibitem{RC-GN:13}
R.~Carli and G.~Notarstefano.
\newblock Distributed partition-based optimization via dual decomposition.
\newblock In {\em {IEEE} Int. Conf. on Decision and Control}, 2013.

\bibitem{RC-GN-LS-DV:15}
R.~Carli, G.~Notarstefano, L.~Schenato, and D.~Varagnolo.
\newblock Analysis of {N}ewton-{R}aphson consensus for multi-agent convex
  optimization under asynchronous and lossy communications.
\newblock In {\em {IEEE} Int. Conf. on Decision and Control}, pages 418--424,
  Osaka, Japan, 2015.

\bibitem{AC-JC:16-auto}
A.~Cherukuri and J.~Cort{\'e}s.
\newblock Initialization-free distributed coordination for economic dispatch
  under varying loads and generator commitment.
\newblock {\em Automatica}, 74:183--193, 2016.

\bibitem{AC-EM-SHL-JC:18-tac}
A.~Cherukuri, E.~Mallada, S.~H. Low, and J.~Cort\'{e}s.
\newblock The role of convexity in saddle-point dynamics: Lyapunov function and
  robustness.
\newblock {\em IEEE Transactions on Automatic Control}, 63(8):2449--2464, 2018.

\bibitem{TD-CB:17}
T.~Doan and C.~Beck.
\newblock Distributed {L}agrangian methods for network resource allocation.
\newblock In {\em {IEEE} Conference on Control Technology and Applications},
  2017.

\bibitem{SF-AI-LS:03}
S.~Friedberg, A.~Insel, and L.~Spence.
\newblock {\em Linear Algebra}.
\newblock Pearson, 4 edition, 2003.

\bibitem{SHM-MJ:18}
S.~Hassan-Moghaddam and M.~Jovanovic.
\newblock On the exponential convergence rate of proximal gradient flow
  algorithms.
\newblock In {\em {IEEE} Int. Conf. on Decision and Control}, 2018.

\bibitem{AH-JH-SB:99}
A.~Hassibi, J.~How, and S.~Boyd.
\newblock A path-following method for solving {BMI} problems in control.
\newblock In {\em {A}merican {C}ontrol {C}onference}, pages 1385--1389, San
  Diego, CA, USA, 1999.

\bibitem{DJ-JX-JM-14}
D.~Jakovetic, J.~Xavier, and J.~Moura.
\newblock Fast distributed gradient methods.
\newblock {\em IEEE Transactions on Automatic Control}, 59(5):1131--1146, 2014.

\bibitem{HKK:02}
H.~Khalil.
\newblock {\em Nonlinear Systems}.
\newblock Prentice Hall, 2002.

\bibitem{EM-CZ-SL:17}
E.~Mallada, C.~Zhao, and S.~Low.
\newblock Optimal load-side control for frequency regulation in smart grids.
\newblock {\em IEEE Transactions on Automatic Control}, 62(12):6294--6309,
  2017.

\bibitem{AM-QL-AR:17}
A.~Mokhtari, Q.~Ling, and A.~Ribeiro.
\newblock An approximate {N}ewton method for distributed optimization.
\newblock {\em IEEE Transactions on Signal Processing}, 65(1):146--161, 2017.

\bibitem{MM-RT:07}
M.~Mozaffaripour and R.~Tafazolli.
\newblock Suboptimal search algorithm in conjunction with polynomial-expanded
  linear multiuser detector for {FDD WCDMA} mobile uplink.
\newblock {\em {IEEE} Transactions on Vehicular Technology}, 56(6):3600--3606,
  2007.

\bibitem{YN:13}
Y.~Nesterov.
\newblock {\em Introductory lectures on convex optimization: A basic course},
  volume~87.
\newblock Springer Science \& Business Media, 2013.

\bibitem{ER-SM:16-ocam}
E.~Ram{\'\i}rez-Llanos and S.~Mart{\'\i}nez.
\newblock Distributed discrete-time optimization algorithms with application to
  resource allocation in epidemics control.
\newblock {\em Optimal Control, Applications and Methods}, 2017.
\newblock To appear. Available at the Wiley Online Library.

\bibitem{YS:03}
Y.~Saad.
\newblock {\em Iterative methods for sparse linear systems}.
\newblock SIAM, 2003.

\bibitem{SYS-MA-LEG:10}
S.~Y. Shafi, M.~Arcak, and L.~E. Ghaoui.
\newblock Designing node and edge weights of a graph to meet {L}aplacian
  eigenvalue constraints.
\newblock In {\em Allerton Conf. on Communications, Control and Computing},
  pages 1016--1023, UIUC, Illinois, USA, 2010.

\bibitem{GS:98}
G.~Stewart.
\newblock {\em Matrix Algorithms Volume 1: Basic Decompositions}.
\newblock SIAM, 1998.

\bibitem{JV-RB:00}
J.~VanAntwerp and R.~Braatz.
\newblock A tutorial on linear and bilinear matrix inequalities.
\newblock {\em Journal of Process Control}, pages 363--385, 2000.

\bibitem{EW-AO-AJ:13P1}
E.~Wei, A.~Ozdaglar, and A.~Jadbabaie.
\newblock A distributed {N}ewton method for network utility maximization, {I}:
  {A}lgorithm.
\newblock {\em IEEE Transactions on Automatic Control}, 58(9):2162--2175, 2013.

\bibitem{EW-AO-AJ:13P2}
E.~Wei, A.~Ozdaglar, and A.~Jadbabaie.
\newblock A distributed {N}ewton method for network utility maximization, {II}:
  {C}onvergence.
\newblock {\em IEEE Transactions on Automatic Control}, 58(9):2176--2188, 2013.

\bibitem{LX-SB:06}
L.~Xiao and S.~Boyd.
\newblock Optimal scaling of a gradient method for distributed resource
  allocation.
\newblock {\em Journal of Optimization Theory \& Applications},
  129(3):469--488, 2006.

\bibitem{FZ-DV-AC-GP-LS:16}
F.~Zanella, D.~Varagnolo, A.~Cenedese, G.~Pillonetto, and L.~Schenato.
\newblock Newton-{R}aphson consensus for distributed convex optimization.
\newblock {\em IEEE Transactions on Automatic Control}, 61(4):994--1009, 2016.

\bibitem{FZ:05}
F.~Zhang.
\newblock {\em The Schur complement and its applications}, volume~4.
\newblock Springer, 2005.

\bibitem{MZ-SM:15-book}
M.~Zhu and S.~Mart{\'\i}nez.
\newblock {\em Distributed Optimization-Based Control of Multi-Agent Networks
  in Complex Environments}.
\newblock Springer-Briefs in Electrical and Computer Engineering. 2015.

\end{thebibliography}

\end{document}